\documentclass[oneside]{amsart}

\usepackage{enumerate}
\usepackage{url}
\usepackage{relsize} 
\usepackage[bbgreekl]{mathbbol}
\DeclareSymbolFontAlphabet{\mathbb}{AMSb} % to ensure \mathbb does not change
\DeclareSymbolFontAlphabet{\mathbbl}{bbold}

\newcommand{\rhcov}{\mathrm{RH}_{\mathrm{cov}}}
\newcommand{\RHc}{\mathrm{RH}_{\mathrm{cont}}}

\newcommand{\rhom}{\mathrm{RHom}}
\newcommand{\free}{\mathrm{free}}
\newcommand{\sym}{\mathrm{Sym}}

\newcommand{\sch}{\mathrm{Sch}}
\usepackage{verbatim}
\usepackage{geometry}
\usepackage{amsfonts}
\usepackage{amssymb}
\usepackage{amsmath}

\usepackage{amsthm}
\usepackage{hyperref}

\usepackage{cleveref}
\newcommand{\schp}{\mathrm{Sch}^{\mathrm{perf}}}
\newcommand{\schpp}{\mathrm{Sch}^{\mathrm{perf}}}
\usepackage{color}
\definecolor{todo}{rgb}{1,0,0}
\definecolor{conditional}{rgb}{0,1,0}
\definecolor{e-mail}{rgb}{0,.40,.80}
\definecolor{reference}{rgb}{.20,.60,.22}
\definecolor{mrnumber}{rgb}{.80,.40,0}
\definecolor{citation}{rgb}{0,.40,.80}

\DeclareMathOperator{\spec}{Spec}

\renewcommand{\hom}{\mathrm{Hom}}

\usepackage{xy}
\input xy

\xyoption{all}

\theoremstyle{definition}
\newtheorem{definition}{Definition}[section]

\newtheorem{construction}[definition]{Construction}
\newtheorem{example}[definition]{Example}
\newtheorem{remark}[definition]{Remark}

\theoremstyle{theorem}
\newtheorem{proposition}[definition]{Proposition}
\newtheorem{lemma}[definition]{Lemma}
\newtheorem{corollary}[definition]{Corollary}

\newtheorem{theorem}[definition]{Theorem}

\renewcommand{\phi}{\varphi}

\newtheoremstyle{named}{}{}{\itshape}{}{\bfseries}{.}{.5em}{#1 \thmnote{#3}}
\theoremstyle{named}

\begin{document}

\newcommand{\shvb}{\mathrm{Shv}_{\bullet}}

\title{The mod $p$ Riemann--Hilbert correspondence and the perfect site}
\author{Akhil Mathew}
%\email{amathew@math.uchicago.edu}
%\address{Department of Mathematics, University of Chicago, 5734 S.~University
%Ave., Chicago, IL 60637}
\date{\today}
\maketitle
\begin{abstract}
The mod $p$ Riemann--Hilbert correspondence (in covariant and
contravariant forms) relates $\mathbb{F}_p$-\'etale sheaves on the spectrum of an
$\mathbb{F}_p$-algebra $R$ and 
Frobenius modules over $R$.  We give an 
exposition of these correspondences using Breen's 
vanishing results on the perfect site. 
\end{abstract}

\section{Introduction}
Let $R$ be a commutative $\mathbb{F}_p$-algebra. 
A \emph{Frobenius module} over $R$ is the datum of an $R$-module $M$ equipped
with a Frobenius-semilinear map $\phi_M : M \to M$. 
Equivalently, a Frobenius module over $R$ is a left module over the twisted
polynomial ring $R[F]$, i.e., the free associative algebra over $R$ on a generator
$F$ with the relations $Fa  = a^p F$; we will typically use the latter notation. 

The mod $p$ Riemann--Hilbert correspondence relates 
$\mathbb{F}_p$-\'etale sheaves on $\spec(R)$ with various types of Frobenius
modules over $R$, and originates in the following result of Katz describing
locally constant constructible sheaves on $\spec(R)$, when $R$ is smooth over a perfect
field. 
\begin{theorem}[{\cite[Prop.~4.1.1]{Katz}}] 
\label{katzthm}
Suppose $R$ is smooth over a perfect field $k$ of characteristic $p$. 
Then there is an equivalence of categories between locally constant
constructible \'etale $\mathbb{F}_p$-sheaves on $\spec(R)$ and 
finitely generated projective $R$-modules $M$ equipped with an isomorphism
$\phi^* M \simeq M$, for $\phi: R \to R$ the Frobenius. 
\end{theorem}

The question of generalizing 
\Cref{katzthm} to more general constructible sheaves, and to more general
$\mathbb{F}_p$-schemes,  has been studied by a number of authors, including
 \cite{EmertonKisin} (who construct a contravariant Riemann--Hilbert
 correspondence) and \cite{ BocklePink, BLRH} (who construct a covariant
 Riemann--Hilbert functor). 
 The purpose of this note is to revisit these results via the vanishing
 results of Breen, \cite{Breen}. 

Let us first describe the contravariant Riemann--Hilbert correspondence
of Emerton--Kisin, \cite{EmertonKisin},  cf.~also \cite{Ohkawa, Schedlmeier} for extensions to
singular schemes. 
We consider the additive group $\mathbb{G}_a$
as a sheaf on $\spec(R)_{\mathrm{et}}$, with a natural left $R[F]$-module
structure. 

\begin{theorem}[{\cite{EmertonKisin}}] 
\label{EKthm}
Suppose $R$ 
is smooth over a perfect field $k$ of characteristic $p$. 
Then the functor\footnote{In \cite{EmertonKisin}, the Riemann--Hilbert functor
is normalized with an additional cohomological shift; this does not affect the
conclusions of this statement.} 
\[  \RHc \stackrel{\mathrm{def}}{=}\mathrm{RHom}_{\mathcal{D}( \spec(R)_{\mathrm{et}}, \mathbb{F}_p)}(-, \mathbb{G}_a): \mathcal{D}^b_{\mathrm{cons}}( \spec(R)_{\mathrm{et}}, \mathbb{F}_p)^{op}  \to
\mathcal{D}( R[F])  \]
is fully faithful, with essential image consisting of those objects 
$X \in \mathcal{D}(R[F])$ which are $t$-bounded and such that, for each
$i \in \mathbb{Z}$: 
\begin{enumerate}
\item $H^i(X)$ is a unit $R[F]$-module: that is, $F$ induces an isomorphism
of $R$-modules
$\phi^* H^i(X) \xrightarrow{\sim} H^i(X)$. 
\item
$H^i(X)$ is finitely generated as an $R[F]$-module. 
\end{enumerate}
\end{theorem} 

The covariant mod $p$ Riemann--Hilbert correspondence
is 
developed in the following form by Bhatt--Lurie
\cite{BLRH}
\footnote{In \cite{BLRH}, the result is stated with
additional boundedness assumptions.},  
and earlier  (in a slightly different form) when $R$ is noetherian by  
B\"ockle--Pink \cite{BocklePink}. 

\begin{theorem}[{\cite[Th.~12.1.5]{BLRH}}, {\cite{BocklePink}}] 
\label{BLthm}
Let $R$ be any $\mathbb{F}_p$-algebra. 
There is a fully faithful, cocontinuous, $t$-exact embedding 
$$\rhcov: \mathcal{D}(\spec(R)_{\mathrm{et}}, \mathbb{F}_p) \to
\mathcal{D}(R[F])$$
of the derived $\infty$-category of \'etale sheaves of $\mathbb{F}_p$-modules on $\spec R$ into
the derived $\infty$-category of $R[F]$-modules. The essential image consists of 
those objects of $X \in \mathcal{D}(R[F])$ such that 
for each $i \in \mathbb{Z}$: 
\begin{enumerate}
\item  $F$ acts isomorphically on $H^i(X)$ (i.e., induces an isomorphism of
$R$-modules $H^i(X) \xrightarrow{\sim} \phi_* H^i(X)$). 
\item Each $x \in H^i(X)$ satisfies an equation of the form
$F^n (x) + a_1F^{n-1}(x) + \dots + a_n x = 0$ for some $a_1, \dots,
a_n \in R$. 
\end{enumerate}

\end{theorem} 

Unlike \Cref{EKthm}, \Cref{BLthm} applies to all 
of
$\mathcal{D}(\spec(R)_{\mathrm{et}}, \mathbb{F}_p)$ (not only to the
constructible objects), works for arbitrary $\mathbb{F}_p$-algebras, and is $t$-exact (hence induces an equivalence on
abelian categories). 
However, 
 the functor $\rhcov$ of \Cref{BLthm} is constructed indirectly,  
 as the left adjoint to an
 explicit solutions functor from $\mathcal{D}(R[F])$ to
 $\mathcal{D}(\spec(R)_{\mathrm{et}}, \mathbb{F}_p)$. 

In this note, we  
give an alternative exposition of \Cref{BLthm} (and its analog for $p$-power
torsion sheaves, \Cref{mainthm} below), based on the vanishing results of 
\cite{Breen}. 
 The strategy is as follows: any sheaf on the \'etale site of
$\spec(R)$ can be pulled back to the  site of \emph{perfect} schemes
over $\spec(R)$, denoted
$\schp_R$, equipped with the \'etale topology, and the pullback functor
$\pi^*$ is fully faithful on derived $\infty$-categories. 
The
results of \cite{Breen} give a fully faithful embedding  of
those objects of $\mathcal{D}(R[F])$ satisfying condition (1) into
$\mathcal{D}(\schp_R, \mathbb{F}_p)$, cf.~\Cref{Breenthm}. 
Via a henselian rigidity argument, one then identifies directly the image of $\pi^*$ in 
$\mathcal{D}(\schp_R, \mathbb{F}_p)$
with the image of 
those objects of $\mathcal{D}(R[F])$ satisfying (1) and (2), from which the
result follows. 

As a consequence, we  obtain  (\Cref{consRH}) the following description of the
covariant Riemann--Hilbert functor: 

\begin{corollary} 
The
covariant Riemann--Hilbert functor $\rhcov$ is given by 
\begin{equation} \label{formulaRH} \rhcov(\mathcal{F)} = 
\mathrm{RHom}_{\mathcal{D}( \schp_R, \mathbb{F}_p)}(\mathbb{G}_a,
\pi^* \mathcal{F})[1], \quad 
\mathcal{F} \in \mathcal{D}(
\spec(R)_{\mathrm{et}}, \mathbb{F}_p) 
,\end{equation}  for $\mathbb{G}_a$ the additive group. 
\end{corollary} 
This formula \eqref{formulaRH} for $\rhcov$ is thus analogous to the formula
in \Cref{EKthm}, except that we additionally pull back to the perfect site. 

We also give an exposition of \Cref{EKthm} using similar methods. 
In particular, we show that a variant of the solutions functor from
\cite{EmertonKisin} gives a fully faithful embedding from 
the subcategory $\mathcal{D}_{\mathrm{unit}}(R[F])$ of $\mathcal{D}(R[F])$ consisting of $t$-bounded unit objects
(i.e., satisfying condition (1) of \Cref{EKthm}, but without any finite
generation assumptions) into the \emph{big} \'etale site
(\Cref{fullfaithfulnessofSol}, which is stated for more general $R$); this fact is closely related to the extension
by Kato
\cite[Prop.~2.1]{Kato1} of the results of \cite{Breen}. Once
again, we can identify those objects in $\mathcal{D}_{\mathrm{unit}}(R[F])$ additionally satisfying (2) 
of \Cref{EKthm} with those sheaves on the big \'etale site satisfying the necessary rigidity
properties to be pulled back from constructible objects on the small \'etale site. 
Along the way, we also show that $\mathcal{D}_{\mathrm{unit}}(R[F])$ is
identified with the derived $\infty$-category of its heart, which is Lyubeznik's
category of $F$-modules or unit $R[F]$-modules \cite{Lyubeznik}; this sharpens results of Ma
\cite{Ma14}. 

It will be convenient to work with unbounded derived
$\infty$-categories of various sorts.  These unbounded
derived $\infty$-categories (for example, of $p$-power torsion \'etale
sheaves on a qcqs $\mathbb{F}_p$-scheme) behave quite well 
because of the presence of bounds on cohomological dimension; we have included
some general results in this direction, using that the relevant abelian
categories are locally regular coherent (\Cref{locregcoh}). 

The arguments of Breen \cite{Breen} rely on some facts about the Steenrod
algebra; purely algebraic proofs via functor homology have also appeared,
cf.~\cite[Lem.~0.3]{FranjouLannesSchwartz} and \cite[Cor.~1.2]{KuhnIII}.
In
the appendix, we explain another argument, suggested by Scholze, which proves the
result using instead the $v$-descent results of \cite{BSWitt}. 

\subsection*{Notation}
We will freely use the theory of Grothendieck prestable $\infty$-categories of
\cite[App.~C]{SAG}. 

We will frequently use the Breen--Deligne resolution in the form of
\cite[Lec.~4]{condensed}. 
Given a site $\mathcal{T}$ with associated topos of sheaves
$\mathrm{Sh}(\mathcal{T})$, and an abelian group object $A \in
\mathrm{Sh}(\mathcal{T})$,
there is a \emph{functorial} resolution $X_\bullet(A)$ of $A$ where each term
$X_i(A)$ is a finite direct sum of terms of the form $\mathbb{Z}[A^j], j \geq
0$. 

For an $\mathbb{F}_p$-algebra $R$, we let $R_{\mathrm{perf}} = \varinjlim_{\phi}
R$ to be the direct limit perfection of $R$. 

\subsection*{Acknowledgments} I am very grateful to Johannes Ansch\"utz, Bhargav Bhatt,
Arthur-C\'esar Le Bras,  
Jacob Lurie, and 
Lucas Mann
 for stimulating discussions and correspondence. 
 The idea of using the results of \cite{Breen} in this
(and other more elaborate) settings was also independently observed by them.   
I also thank Matthew Emerton for a helpful conversation about
\cite{EmertonKisin}, and Linquan Ma for exchanges about $F$-finiteness. 
 Special thanks to Peter Scholze
 for suggesting the argument in the appendix. 
 I heartily thank the referee for many helpful comments and corrections. 
 This
work was done while the author was a Clay Research Fellow and while the author
was supported by the National Science Foundation (\#2152235). 

\section{Preliminaries on the derived categories}

\newcommand{\ptors}{p-\mathrm{tors}}
Let $\mathcal{T}$ be a site (we use the conventions of \cite[Tag
00VG]{stacks-project}); 
moreover, suppose $\mathcal{T}$ is coherent, i.e., fiber products exist in $\mathcal{T}$
and any cover has a finite refinement. 
This implies that abelian sheaf cohomology on $\mathcal{T}$ commutes with
filtered colimits \cite[Exp.~VI, Cor.~5.2]{SGA4}. 

We consider the unbounded derived $\infty$-category $\mathcal{D}(\mathcal{T})$ of the Grothendieck abelian category of sheaves of abelian groups on $\mathcal{T}$, or
equivalently 
\cite[Cor.~2.1.2.3]{SAG}
the $\infty$-category of hypercomplete sheaves on $\mathcal{T}$ with values in
$\mathcal{D}(\mathbb{Z})$. 
The presentable, stable $\infty$-category $\mathcal{D}(\mathcal{T})$ is equipped
with a natural $t$-structure, whose connective objects
$\mathcal{D}(\mathcal{T})^{\leq 0}$ form a Grothendieck prestable
$\infty$-category in the sense of \cite[App.~C]{SAG}. 
We let $\mathcal{D}(\mathcal{T})_{\ptors} \subset
\mathcal{D}(\mathcal{T})$ denote the subcategory of $p$-power torsion
objects: i.e., those $X$ with $X[1/p] =0 $. Then $\mathcal{D}(\mathcal{T})_{\ptors}$
inherits a natural $t$-structure from 
$\mathcal{D}(\mathcal{T})$, and $\mathcal{D}(\mathcal{T})_{\ptors}^{\leq 0}$ is
Grothendieck prestable by \cite[Prop.~C.5.2.1]{SAG}. 
We say that an object $t \in \mathcal{T}$ has {mod $p$ cohomological
dimension $\leq d$} if for every sheaf $\mathcal{F}$ of $\mathbb{F}_p$-vector
spaces (and thus more generally every $p$-power torsion sheaf) on $\mathcal{T}$,
the groups $H^i(t, \mathcal{F}) =0 $ for $i > d$.

\begin{proposition} 
Suppose moreover that $\mathcal{T}$ has a basis $\mathcal{B}$ consisting of objects 
of mod $p$ cohomological
dimension $\leq d$ for some fixed $d$. 
Then: 
\begin{enumerate}
\item  
$\mathcal{D}(\mathcal{T})_{\ptors}$ is left complete. 
\item
$\mathcal{D}(\mathcal{T})_{\ptors}$ is the derived $\infty$-category
of the heart (of $p$-power torsion abelian sheaves on $\mathcal{T}$). 
\item 
The $\infty$-category 
$\mathcal{D}(\mathcal{T})_{\ptors}^{\leq 0}$
of connective objects in 
$\mathcal{D}(\mathcal{T})_{\ptors}$
is compactly generated by the
objects $\mathbb{Z}/p^n [h_t], t \in \mathcal{B}, n \geq 0$, for $h_t$ the sheafification of the
representable presheaf (of sets) defined by $t$. Moreover, $\mathcal{D}(\mathcal{T})_{\ptors}$
is generated as a localizing subcategory by the 
$\mathbb{Z}/p^n [h_t], t \in \mathcal{B}, n \geq 0$. 
\end{enumerate}
\label{derivedTptors}
\end{proposition} 
\begin{proof} 
Fix $X \in \mathcal{D}(\mathcal{T})_{\ptors}$. In
$\mathcal{D}(\mathcal{T})$, the object $X$ is the inverse limit of
its Postnikov tower thanks to \cite[Prop.~2.10]{CM21}, since by our
assumptions the objects $\mathbb{Z}[h_t], t\in \mathcal{B} $ have cohomological
dimension $\leq d$ with $p$-power torsion coefficients. 
Therefore, we find that for any $t \in \mathcal{B}$, the sections $R \Gamma(t,
X) = \mathrm{RHom}_{\mathcal{D}(\mathcal{T})}(\mathbb{Z}[h_t], X) \in
\mathcal{D}(\mathbb{Z})$ is $p$-power torsion: this follows by passage up the
Postnikov tower of $X$, using the cohomological dimension bound on $t$ to see
that the limit stabilizes in any range of degrees. Moreover, the same argument shows that the construction $X \mapsto R \Gamma(t, X)$
therefore commutes with filtered colimits, from
$\mathcal{D}(\mathcal{T})_{\ptors} $ to $\mathcal{D}(\mathbb{Z})$. 
This implies that 
$\{\mathbb{Z}/p^n [h_t], t \in \mathcal{B}, n \geq 0\}$ are compact in 
$\mathcal{D}(\mathcal{T})_{\ptors}$. 

We now show that $\mathcal{D}(\mathcal{T})_{\ptors}$ is 
0-complicial \cite[Def.~C.5.3.1]{SAG}, i.e.,  for any object $X \in
\mathcal{D}(\mathcal{T})^{\leq 0}_{\ptors}$, there exists a discrete object $X'$ 
(which will in fact be a sum of $\mathbb{Z}/p^n[h_t]$'s)
and a map $X' \to X$
inducing a surjection on $H^0$. 
For every $t \in \mathcal{B}$, any map $\mathbb{Z}[h_t] \to X$ necessarily
factors over $\mathbb{Z}/p^n[h_t] \to X$ for some $n \gg 0$, by the above. 
It follows that the 
natural map
\[ \bigoplus_{n \geq 0, t \in \mathcal{B}, f: \mathbb{Z}/p^n[h_t] \to X}
\mathbb{Z}/p^n[h_t] \to X  \]
induces a surjection on $H^0$, and the former is discrete. This gives the
desired 0-complicial statement. Since $\mathcal{D}(\mathcal{T})_{\ptors}$ is clearly left separated, we find by 
\cite[Prop.~C.5.4.5]{SAG} (see also \cite[Rem.~C.5.4.11]{SAG}) that
$\mathcal{D}(\mathcal{T})_{\ptors}$
is the derived $\infty$-category of its heart. 
Moreover, \cite[Cor.~C.2.1.7]{SAG} now shows that the $\mathbb{Z}/p^n[h_t], t \in
\mathcal{B}, n \geq 0$ generate $\mathcal{D}(\mathcal{T})_{\ptors}^{\leq 0}$ under
colimits (and the full 
$\mathcal{D}(\mathcal{T})_{\ptors}$ as a localizing subcategory).

Since the abelian category of $p$-power torsion sheaves on $\mathcal{T}$ 
is generated by the objects $\mathbb{Z}/p^n[h_t], t \in \mathcal{B}, n \geq 0$,
we find from \cite[Ex.~2.21]{CM21} that 
$\mathcal{D}(\mathcal{T})_{\ptors}$
is Postnikov complete. 
\end{proof}

\begin{example}[The small \'etale site] 
\label{smalletalesite}
Let $X$ be any qcqs $\mathbb{F}_p$-scheme. The \'etale site $X_{\mathrm{et}}$
has a basis  of objects of
mod $p$ cohomological dimension $\leq 1$. 
In fact, for any affine $\mathbb{F}_p$-scheme $U$, and any sheaf $\mathcal{F}$ of $\mathbb{F}_p$-vector
spaces on $U_{\mathrm{et}}$, 
we have $H^i(U, \mathcal{F}) = 0$ for $i > 1$; we may therefore take the
affines as the desired basis. 

To see this claim, we may assume that $\mathcal{F}$ is constructible by writing 
$\mathcal{F}$ as a filtered colimit of constructible sheaves \cite[Tag
03SA]{stacks-project} and using \cite[Tag 073E]{stacks-project}. 
By \cite[Tag 09YU]{stacks-project}, we may then assume that $\mathcal{F}$ is pulled
back from the spectrum of a finitely generated $\mathbb{F}_p$-algebra. 
Using \cite[Cor.~5.10,
Exp.~VII]{SGA4}, we are thus reduced to proving the cohomological dimension
bound in the noetherian case, where it is 
\cite[Th.~5.1, Exp.~X]{SGA4}. 

Thus, we obtain the above conclusions for $\mathcal{D}(
X_{\mathrm{et}})_{\ptors}$. Namely, 
$\mathcal{D}(
X_{\mathrm{et}})_{\ptors}$
is the derived $\infty$-category of the abelian category of $p$-power torsion
sheaves of abelian groups on $X$, and it is
left complete and compactly generated. 
\end{example} 

\begin{example}[The perfect  site] 
\label{perfproet}
Let $X$ be a qcqs perfect $\mathbb{F}_p$-scheme. 
Let $\schpp_X$ (also written $\schpp_R$ if $X = \spec(R)$ is affine) be the site of all perfect qcqs $X$-schemes, equipped with
the \'etale topology. 
Again, as in \Cref{smalletalesite}, the basis of affines shows that the above
conclusions hold: $\mathcal{D}(\schpp_X)_{\ptors}$ is the derived
$\infty$-category 
of the abelian category of $p$-power torsion sheaves of abelian groups on $\schpp_X$, and it is left
complete and compactly generated. 
\end{example}

We let $(\pi^*, \pi_*)$ denote 
the natural adjunction
\[ (\pi^*, \pi_*): \mathcal{D}( X_{\mathrm{et}}) \rightleftarrows
\mathcal{D}( \schp_X).  \]
from the derived $\infty$-category of the \'etale site to the derived
$\infty$-category of 
$\schp_R$, arising from the inclusion of sites $X_{\mathrm{et}}
\subset \schp_X$. 

\begin{proposition} 
\label{restrictedadjunction}
For any $\mathcal{F} \in \mathcal{D}(X_{\mathrm{et}})$, the adjunction
map 
$\mathcal{F} \to \pi_* \pi^* \mathcal{F}$ is an equivalence.   
Moreover, $\pi_*$ is cocontinuous. 
\end{proposition} 
\begin{proof} 
For discrete objects, the result 
follows classically from the comparisons between big and small sites of \cite[Exp.~VII,
Sec.~4]{SGA4} 
or \cite[Tag 00XU]{stacks-project}. 
The adjoint functors on hearts pass to adjoint functors on derived
$\infty$-categories
(note that both $\pi^*, \pi_*$ are cocontinuous and exact on abelian
		categories), whence the claim,
cf.~\cite[Th.~C.5.4.9]{SAG}. 
This also follows from \cite[Prop.~7.1]{CM21}. 
\end{proof} 

\begin{remark} 
\label{imageofpi}
We can identify the image of the fully faithful embedding $\pi^*$ of
\Cref{restrictedadjunction}, at least on $p$-power torsion objects. 
An object $\mathcal{F} \in \mathcal{D}(\schpp_X)_{\ptors}$ belongs to
the image of $\pi^*$ if and only if: 
\begin{enumerate}
\item As a functor on rings, $\mathcal{F}$ commutes with
filtered colimits. 
\item
For any local homomorphism of  perfect, strictly henselian local
rings $R_1 \to R_2$ over $X$, the map $\mathcal{F}(R_1) \to \mathcal{F}(R_2)$ is
an equivalence. 
\end{enumerate}
In fact, the pullback of any object of
$\mathcal{D}(X_{\mathrm{et}})_{\ptors}$ clearly has these
properties, since \'etale cohomology commutes with filtered colimits of rings
\cite[Cor.~5.8, Exp.~VII]{SGA4}
and because of mod $p$ cohomological dimension
$\leq 1$ for affines.
Conversely, if $\mathcal{F}$ has these properties, then the adjunction map 
$\pi^* \pi_* \mathcal{F} \to \mathcal{F}$ is an equivalence on strictly
henselian local perfect rings over $X$ (already by (2) alone), and hence in
general by \'etale descent 
of both $\pi^* \pi_* \mathcal{F}, \mathcal{F}$
and by (1) to identify the stalks of both sides. 
\end{remark}

We include for completeness some additional results on the structure of
$\mathcal{D}(X_{\mathrm{et}})_{\ptors}$ which will not be used in the construction
of the Riemann--Hilbert correspondence. In particular, we identify the compact
objects precisely, using the following definition. \begin{definition} 
\label{locregcoh}
We say that a Grothendieck abelian category $\mathcal{A}$ is \emph{locally
regular coherent} if: 
\begin{enumerate}
\item $\mathcal{A} $ is compactly generated. 
\item The subcategory $\mathcal{A}_0 \subset \mathcal{A}$ of compact objects is
closed under finite limits (so it is an abelian subcategory). 
\item Given
$X \in \mathcal{A}_0$, there exists an integer $d_X$ such that one has $\mathrm{Ext}^i_{\mathcal{A}}(X, Y) =0$ for $i
>  d_X$ and for all $Y \in \mathcal{A}_0$. 
\item There exists an integer $d$ and a subcategory $\mathcal{A}_0' \subset \mathcal{A}_0$ such 
that if $X \in \mathcal{A}_0', Y \in \mathcal{A}_0$, then
$\mathrm{Ext}^i_{\mathcal{A}}(X, Y) = 0$ for $i > d$. 
The subcategory $\mathcal{A}_0'$ generates $\mathcal{A}$ under colimits. 
\end{enumerate}
\end{definition} 

The category of modules over any coherent ring such that any finitely presented module is of finite
projective dimension (e.g., the infinite-dimensional polynomial algebra
$\mathbb{F}_p[x_1, x_2, x_3, \dots ]$)
is an example of a locally regular coherent Grothendieck abelian category (take
$\mathcal{A}_0'$ to be the category of finitely generated free modules).

\begin{proposition} 
Suppose $\mathcal{A}$ is a locally regular coherent Grothendieck abelian
category. Then: 
\begin{enumerate}
\item $\mathcal{D}(\mathcal{A})$ is Postnikov complete.  
\item The compact objects of $\mathcal{D}(\mathcal{A})$ consist of those objects
which are $t$-bounded and all of whose cohomology groups belong to
$\mathcal{A}_0$. 
\end{enumerate}
\label{locregcritabelian}
\end{proposition} 
\begin{proof} 
Let us first show that the objects of $\mathcal{A}_0 \subset \mathcal{A}$ are
pseudo-coherent (also called almost compact), i.e.,  they are compact in each truncation
$\mathcal{D}(\mathcal{A})^{\geq -n}$. In fact, by 
\cite[Cor.~C.6.5.9]{SAG}, the Grothendieck prestable $\infty$-category
$\mathcal{D}(\mathcal{A})^{\leq 0}$ is coherent in the sense of
\cite[Def.~C.6.5.1]{SAG}, and in particular 
any object of $\mathcal{A}$ can be written as a filtered colimit of objects
in $\mathcal{A}$ which are 
pseudo-coherent objects in $\mathcal{D}(\mathcal{A})^{\leq 0}$,
cf.~\cite[Prop.~C.6.5.6]{SAG}; by compactness in the heart, this means that any
object of $\mathcal{A}_0$ is necessarily pseudo-coherent. 
Since the subcategory of pseudo-coherent objects 
of $\mathcal{D}(\mathcal{A})^{\leq 0}$
is closed under finite limits (again by coherence, \cite[Def.~C.6.5.1]{SAG}) and
therefore under extensions,\footnote{An extension is obtained as a fiber of a
map to the suspension.} we conclude that $\mathcal{A}_0$ is closed under
extensions.

Now by pseudo-coherence, it follows that if $X \in \mathcal{A}_0, Y \in
\mathcal{A}$, then $\mathrm{Ext}^i_{\mathcal{A}}(X, Y) =0 $ for $i > d_X$.
Moreover, if $X \in \mathcal{A}_0'$, then the vanishing holds for $i > d$.
This also follows by pseudo-coherence of $X$, since we have assumed the
vanishing for $Y \in \mathcal{A}_0$. 

The first claim follows from \cite[Ex.~2.21]{CM21} since we have the generating
objects $\mathcal{A}_0'$, which we have just seen to be of cohomological
dimension $\leq d$. 

Given $X \in \mathcal{A}_0$, we claim that $X$ is actually compact in
$\mathcal{D}(\mathcal{A})$. In fact, we have seen that $X$ is pseudo-coherent.
Since $X$ has finite cohomological dimension by assumption, and since
(as just proved) Postnikov
towers converge in $\mathcal{D}(\mathcal{A})$, we conclude that $X$ is compact
in $\mathcal{D}(\mathcal{A})$. Consequently, any $t$-bounded object in
$\mathcal{D}(\mathcal{A})$ all of whose cohomology groups belong to
$\mathcal{A}_0$ is compact in 
$\mathcal{D}(\mathcal{A})$; the converse claim follows now since $\mathcal{A}_0$
yields a set of compact generators
of $\mathcal{D}(\mathcal{A})$ as a localizing subcategory, whence any compact
object belongs to the thick subcategory generated by $\mathcal{A}_0$. 
\end{proof}

\begin{proposition} 
\label{compactobjofderivedcat}
For any qcqs $\mathbb{F}_p$-scheme $X$,
the abelian category of $p$-power torsion \'etale sheaves on $X$ is locally
regular coherent. 
Thus, 
the compact objects of 
$\mathcal{D}(
X_{\mathrm{et}})_{\ptors}$ are precisely the objects which are bounded in
the $t$-structure and such that each cohomology group is constructible. 
\end{proposition} 
\begin{proof} 
The abelian category $\mathcal{A}$ of $p$-power torsion sheaves on
$X_{\mathrm{et}}$ is compactly generated
by the constructible $p$-power torsion sheaves, which form an abelian
subcategory $\mathcal{A}_0 \subset \mathcal{A}$, cf.~\cite[Tags
05BE, 095M]{stacks-project} or \cite[Exp.~IX, sec.~2]{SGA4}. For each affine scheme $U$ and \'etale map $j: U \to X$, the object
$j_!(\mathbb{Z}/p^n)$ is of cohomological dimension $\leq 2$ and these objects
generate $\mathcal{A}$ under colimits. 

In light of \Cref{locregcritabelian} (and \Cref{derivedTptors}) with
$\mathcal{A}_0'$ the collection $\left\{j_!(\mathbb{Z}/p^n), j: U \to X
\ \text{\'etale}\right\}$, 
it suffices to show that if $\mathcal{F}$ is a constructible $p$-power
torsion sheaf 
 on $X_{\mathrm{et}}$, then $\mathcal{F}$ has finite
cohomological dimension in $\mathcal{A}$, i.e.,  
there exists $n$ such that $\mathrm{Ext}^i(\mathcal{F}, \mathcal{G}) =0 $ for
any constructible
$p$-power torsion sheaf
$\mathcal{G}$ on $X_{\mathrm{et}}$ and $i > n$. By \cite[Exp.~IX, Prop.~2.5]{SGA4}, we can assume that $\mathcal{F} = f_!
\mathcal{F}_0$ for $f: U \to X$ a locally closed constructible embedding
and $\mathcal{F}_0$ a locally constant constructible sheaf  on $U$. 
Now $Rf^!$ has finite cohomological dimension: in fact, this follows from
factoring $f$ into an open embedding and a closed embedding and using 
\cite[Exp.~X, Th.~5.1]{SGA4} (and \Cref{smalletalesite} in the non-noetherian
case), which shows that any qcqs 
 $\mathbb{F}_p$-scheme has bounded mod $p$ cohomological
dimension.

By duality, the result now follows because (writing $\mathcal{F}_0^{\vee} =
\underline{\mathrm{RHom}}( \mathcal{F}_0, \mathbb{Z})$)
$\mathrm{Ext}^i( \mathcal{F}, \mathcal{G}) = H^i(U,  \mathcal{F}_0^{\vee}
\otimes^{\mathbb{L}}
Rf^! \mathcal{G})$ and $U$ has bounded $\mathbb{F}_p$-cohomological dimension
by \cite[Exp.~X, Th.~5.1]{SGA4} again. 
\end{proof}

\newcommand{\wfr}{{W(R)[F^{\pm 1}]}}

\section{$\wfr$-modules}

Let $R$ be a perfect $\mathbb{F}_p$-algebra. 
The purpose of this section is to introduce the ring $\wfr$ and review the
notion of an \emph{algebraic} $\wfr$-module, following \cite{BLRH}.  
Let $W(R)$ denote the ring of Witt vectors of $R$ with Witt vector Frobenius
$\phi: W(R) \to W(R)$. 

\begin{construction}[{The ring $\wfr$}] 
Let $W(R)[F^{\pm 1}]$  denote the 
\emph{noncommutative} ring over $W(R)$ 
defined as follows: $\wfr$ is spanned freely as  a left module over $W(R)$ by the
powers $F^i, i \in \mathbb{Z}$, with the commutation rule $F^i a = \phi^i(a)
F^i$ for $a \in W(R), i \in \mathbb{Z}$. 
We will also consider the quotient $R[F^{\pm 1}] = \wfr/p$; a module over
$\wfr/p$ is a perfect Frobenius module over $R$ in the sense of \cite{BLRH}.

We will consider the derived $\infty$-category $\mathcal{D}(\wfr)$ of
left $\wfr$-modules, and the full subcategory $\mathcal{D}(\wfr)_{\ptors} \subset
\mathcal{D}(\wfr)$ consisting of objects $M$ with $M[1/p]= 0$.

\end{construction}

\begin{remark} 
\label{relativetensor}
Given  $M \in \mathcal{D}(\wfr)$  and a right
$\wfr$-module $M'$ (more generally, $M'$ could belong to the appropriate
derived $\infty$-category), we can form the relative derived tensor product $M'
\otimes^{\mathbb{L}}_{\wfr} M \in \mathcal{D}(\mathbb{Z})$,
cf.~\cite[Sec.~4.4]{HA} for a very general account.

We will need the following example of this later. 
Given a \emph{perfect} $R$-algebra $R'$, we can also make $W(R')$ into a
\emph{right} $\wfr$-module, where the action of $F$ is by $x \mapsto
\phi^{-1}(x)$. 
If $M \in \mathcal{D}(\wfr)$, then we can describe 
$W(R') \otimes^{\mathbb{L}}_{\wfr} M$ as a cofiber
\begin{equation} W(R') \otimes^{\mathbb{L}}_{\wfr } M  =\mathrm{cofib}(F -1: W(R') \otimes_{W(R)}^{\mathbb{L}}
M  \to W(R') \otimes_{W(R)}^{\mathbb{L}} M),
\label{cofibexpr}
\end{equation}
where $W(R') \otimes^{\mathbb{L}}_{W(R)} M
= W(R')[F^{\pm 1}] \otimes^{\mathbb{L}}_{W(R)[F^{\pm 1}]} M \in \mathcal{D}(
W(R')[F^{\pm 1}])
$ is the extension of scalars. 
In fact, this follows from the short exact sequence of right $W(R)[F^{\pm
1}]$-modules,
\[ 0 \to W(R')[F^{\pm 1}] \xrightarrow{(F-1) \cdot} W(R')[F^{\pm 1}] \to  W(R') \to 0.   \]
\end{remark}

\begin{definition}[{Cf.~\cite[Def.~2.4.1]{BLRH}}] 
A discrete $\wfr$-module $M$ is \emph{algebraic} if $M[1/p] = 0$ and if 
any $x \in M$ satisfies an equation of the form $F^n x  + a_1 F^{n-1} x  +
\dots + a_n x = 0$ for some $a_1, \dots, a_n \in W(R)$ (equivalently, if the
$W(R)$-submodule spanned by $\left\{F^i x\right\}_{i \geq 0}$ is finitely
generated).

An object of $\mathcal{D}(\wfr)$ is \emph{algebraic} if the
cohomology modules are algebraic. We let $\mathcal{D}_{\mathrm{alg}}(\wfr)
\subset \mathcal{D}(\wfr)$ be the full subcategory spanned by the algebraic
objects. 
\end{definition} 

\begin{proposition} 
\label{algebraicmodulesstability}
\begin{enumerate}
\item If $0 \to M' \to M \to M'' \to 0$ is a short exact sequence of
$\wfr$-modules, then  $M$ is algebraic if and only if $M', M''$ are algebraic. 
\item The collection of algebraic $\wfr$-modules is closed under arbitrary 
colimits. 
\end{enumerate}
\end{proposition} 
\begin{proof} 
This follows from \cite[Prop.~9.5.5]{BLRH}; note that an $\wfr$-module $N$ with
$N[1/p]=0$ is algebraic if and only if for each $n$, the $p^n$-torsion $N[p^n]$
is an algebraic Frobenius $W_n(R)$-module in the sense of
\cite[Def.~9.5.2]{BLRH}. 
\end{proof}

\begin{corollary} 
The subcategory $\mathcal{D}_{\mathrm{alg}}(\wfr) \subset \mathcal{D}(\wfr)$ is a localizing
subcategory. 
\end{corollary}
\begin{proof} 
The subcategory 
$\mathcal{D}_{\mathrm{alg}}(\wfr) \subset \mathcal{D}(\wfr)$
 is a localizing
subcategory thanks to \Cref{algebraicmodulesstability}. 
\end{proof} 

Moreover, the $t$-structure on $\mathcal{D}(\wfr)$ restricts to a $t$-structure
on 
$\mathcal{D}_{\mathrm{alg}}(\wfr)$, which is clearly left and right complete,
and the subcategory $\mathcal{D}_{\mathrm{alg}}(\wfr)^{\leq 0}$ of connective
objects is Grothendieck prestable,
\cite[Prop.~C.5.2.1]{SAG}. 

\begin{example} 
\label{examplesofalgebraic}
For any monic element $T = F^n + a_1 F^{n-1} + \dots + a_n \in \wfr$ and $j \geq 0$,  the quotient
$\wfr/(p^j, \wfr T)  $ is algebraic. 
This can be checked directly or follows from \cite[Prop.~4.2.1]{BLRH} (first
reducing to $j=1$), since
(in the terminology of \emph{loc.~cit.}) $R[F^{\pm 1}]/R[F^{\pm 1}] T$ is the
perfection of the Frobenius module $R[F]/R[F] T$, which is finite free over
$R$, whence the former is holonomic in the sense of \cite[Def.~4.1.1]{BLRH}. 
\end{example}

\begin{proposition}
The $\infty$-category $\mathcal{D}_{\mathrm{alg}}(\wfr)^{\leq 0} \subset
\mathcal{D}_{\mathrm{alg}}(\wfr)$  is compactly generated 
by the objects of \Cref{examplesofalgebraic}. 
Consequently, 
$\mathcal{D}_{\mathrm{alg}}(\wfr) $ is generated as a localizing subcategory by
the objects
of 
\Cref{examplesofalgebraic}. 
Moreover, 
$\mathcal{D}_{\mathrm{alg}}(\wfr) $ is the derived $\infty$-category of the
abelian category of algebraic Frobenius modules. 
\label{algebraicgenerators}
\end{proposition}
\begin{proof} 
The objects of 
\Cref{examplesofalgebraic} are evidently compact, since the sequence $(p^j, T)
\in \wfr$ is regular. 
Let $M \in \mathcal{D}(\wfr)^{\leq 0}$ be algebraic. We claim that 
for any $x \in H^0(M)$, there exists a monic polynomial $T = F^n + a_1 F^{n-1} +
\dots + a_n \in \wfr$ and a map $\wfr/(p^n, T) \to M$ carrying the unit to $x$. 
In fact, since $x$ is annihilated by some $T$, we obtain 
a map $\wfr/T \to M$ in $\mathcal{D}(\wfr)$; this map is annihilated by some
power of $p$ by compactness of $\wfr/T \in \mathcal{D}( \wfr)^{\leq 0}$, and thus factors over 
$\wfr/(p^n, T)$.
This proves the claim, whence the result follows in light of
\cite[Cor.~C.2.1.7]{SAG} and \cite[Rem.~C.5.4.11]{SAG}. 
\end{proof} 

\begin{proposition} 
\label{extofscalars}
Let $R \to R'$ be any map of perfect $\mathbb{F}_p$-algebras. The extension of
scalars functor
$W(R')[F^{\pm 1}] \otimes^{\mathbb{L}}_{W(R)[F^{\pm 1}]} (-): \mathcal{D}(
W(R)[F^{\pm 1}]) \to 
\mathcal{D}(W(R')[F^{\pm 1}])$ carries algebraic objects to algebraic objects. 
\end{proposition} 
\begin{proof} 
This follows from 
\Cref{algebraicgenerators}, since the generators of \Cref{algebraicgenerators}
are preserved under extension of scalars and extension of scalars is
cocontinuous. 
\end{proof} 

We next observe that restriction of scalars along a closed embedding is fully
faithful on $p$-power torsion objects of $\mathcal{D}(\wfr)$; this is an analog
of Kashiwara's theorem, cf.~\cite[Th.~5.3.1]{BLRH}. 

\begin{proposition} 
Let $R \to R'$ be any surjection of perfect $\mathbb{F}_p$-algebras. The
restriction of scalars functor $\mathcal{D}(W(R')[F^{\pm 1}])_{\ptors} \to
\mathcal{D}(W(R)[F^{\pm 1}])_{\ptors}$ is fully faithful.  The
restriction of scalars functor carries algebraic objects
into algebraic objects. 
\end{proposition} 
\begin{proof} 
The last assertion is evident, so it suffices to show that 
if $M' \in \mathcal{D}(W(R')[F^{\pm 1}])_{\ptors} $, the adjunction map 
$W(R')[F^{\pm 1}] \otimes^{\mathbb{L}}_{W(R)[F^{\pm 1}]} M' \to M'$ is an equivalence. This
reduces by taking colimits and Postnikov towers to the case where $M'$ is
discrete and annihilated by $p$. As objects of $\mathcal{D}(R')$, the adjunction
map is $R' \otimes^{\mathbb{L}}_R M' \to M'$, which is an equivalence since $R'
\otimes^{\mathbb{L}}_R R' \to R'$ is an equivalence,
cf.~\cite[Lem.~11.10]{BSWitt}.  
\end{proof}

\begin{proposition} 
Let $C$ be a perfect $R$-algebra. 
Suppose $C$ is integral over $R$. Consider $C$ as a left $\wfr$-module with $F$
acting by Frobenius; then $C$ is algebraic. 
\label{integralisalgebraic}
\end{proposition} 
\begin{proof} 
Fix an element $x \in C$; we need to show that it is annihilated by a monic
polynomial in $F$ in $R[F^{\pm 1}]$. 
Since $C$ is a filtered colimit of perfections of finite, finitely presented
$R$-algebras, we may reduce to the case when $C$ is of this form. 
Now, we may descend $C$ as follows: there exists a finitely generated subalgebra
$R_0 \subset R$, a finite $R_0$-algebra $C_0$, and an isomorphism $C = (C_0
\otimes_{R_0} R)_{\mathrm{perf}}$. Moreover, we can arrange that $x$ descends to an
element $x_0 \in C_0$. Since $C_0$ is a finite $R_0$-module, the $R_0$-submodule
of $C_0$
generated by the Frobenius iterates of $x_0$
is finitely generated, whence we obtain the algebraicity claim. See
also \cite[Prop.~4.2.1]{BLRH} for a more general result: any holonomic $R[F^{\pm 1}]$-module is algebraic. 
\end{proof}

In the remainder of this section, we observe that the category of algebraic $\wfr$-modules is locally regular
coherent in the sense of \Cref{locregcoh}. This will not be needed for the proof of
\Cref{mainthm} (and in fact also follows from \Cref{mainthm} together with
\Cref{compactobjofderivedcat}), but we include an independent argument (relying
on the results on holonomic $R[F^{\pm 1}]$-modules from 
\cite{BLRH}) for completeness.

An $R[F^{\pm 1}]$-module is defined to be \emph{holonomic} in
\cite[Def.~4.1.1]{BLRH} if 
it is obtained via extension of scalars from an $R[F]$-module which is finitely
presented as an $R$-module. 
The subcategory of holonomic modules
is an abelian  subcategory of the category of $R[F^{\pm 1}]$-modules and is
stable under extensions, \cite[Cor.~4.3.3]{BLRH}. 
The category of algebraic $R[F^{\pm 1}]$-modules 
(an $R[F^{\pm 1}]$-module is algebraic if it is algebraic as a $\wfr$-module)
is
compactly generated, with the holonomic modules as the subcategory of compact
objects \cite[Th.~4.2.9]{BLRH} (where we use also that the subcategory of
holonomic modules is idempotent-complete since it is abelian). We extend the definition to $\wfr$-modules as follows. 

\begin{definition} 
A (discrete) $\wfr$-module $M$ is said to be \emph{holonomic} if it is
algebraic and if it is compact in
the category of 
algebraic $\wfr$-modules. 
\end{definition}

\begin{remark} 
\label{holonomicgenerators}
Since 
$\mathcal{D}_{\mathrm{alg}}(\wfr)^{\leq 0}$ is compactly generated by
\Cref{algebraicgenerators}, it follows
that the category of algebraic $\wfr$-modules (the heart) is compactly
generated, with compact objects the cokernels of maps between finite direct
sums of the objects of \Cref{examplesofalgebraic}.  
It also follows that an algebraic $\wfr$-module is holonomic if and only if it
is finitely presented as a module over $\wfr$. 
\end{remark}

\begin{proposition} 
If $R \to R'$ is any map of perfect $\mathbb{F}_p$-algebras, the extension of
scalars functor $W(R')[F^{\pm 1}] \otimes^{\mathbb{L}}_{\wfr} (-):
\mathcal{D}_{\mathrm{alg}}(\wfr) \to \mathcal{D}_{\mathrm{alg}}(W(R')[F^{\pm
1}])$ is $t$-exact. 
\label{texactbasechange}
\end{proposition} 
\begin{proof} 
Since the functor is clearly right $t$-exact, it suffices to show that if $M$ is
a discrete algebraic $\wfr$-module, then $W(R')[F^{\pm 1}]
\otimes^{\mathbb{L}}_{\wfr} M$ is discrete. This follows from
\cite[Th.~3.5.1]{BLRH} (and reduction by d\'evissage to the case $pM =0 $). 
\end{proof} 

\begin{lemma} 
Suppose that $R$ is the perfection of a  finitely generated 
algebra over $\mathbb{F}_p$.
Then: 
\begin{enumerate}
\item  
The category of algebraic $\wfr$-modules is 
locally noetherian (cf.~\cite[Sec.~C.6.8]{SAG} for a treatment) with the
holonomic modules as noetherian objects; the holonomic modules are also precisely the algebraic
$\wfr$-modules which are finitely generated as $\wfr$-modules. 
\item
Any
$\wfr$-module annihilated by a power of $p$ has finite projective dimension. 
\end{enumerate}
\label{lem:noetherian}
\end{lemma} 
\begin{proof} 

To prove local noetherianity, we show that if 
$M$ is an algebraic $\wfr$-module such that 
$M$ is finitely generated as a module over $\wfr$, then any $\wfr$-submodule $M'
\subset M$ is also finitely generated. 
In fact, $M/pM$ is a finitely generated algebraic $R[F^{\pm 1}]$-module, and the
category of algebraic $R[F^{\pm 1}]$-modules is 
locally noetherian
\cite[Prop.~4.3.1 and 3.4.3]{BLRH}, with the holonomic (or in this case finitely generated
algebraic) modules as noetherian objects. 
Intersecting the (finite) $p$-adic filtration on $M$ with $M'$, we see that $M'$ admits a
finite filtration by finitely generated $R[F^{\pm 1}]$-modules, and hence is
finitely generated over $\wfr$. 
It follows that the category of finitely generated algebraic $\wfr$-modules is
an abelian category, and is precisely the subcategory of holonomic $\wfr$-modules. 
This gives the local noetherianity and proves (1). 
%Since the category 
%of algebraic $\wfr$-modules is thus locally noetherian, it follows that any holonomic
%$\wfr$-module is pseudo-coherent in the derived $\infty$-category of algebraic
%$\wfr$-modules, cf.~\cite[Prop.~C.6.9.3, Ex.~C.9.5]{SAG}. 

Finally, we claim that 
any $\wfr$-module $N$
which is annihilated by a power of $p$ has finite projective dimension as a $\wfr$-module. It suffices to show any
 $R[F^{\pm 1}]$-module has finite projective
dimension as an $R[F^{\pm 1}]$-module; this follows from \cite[Rem.~3.1.8
and 3.2.8]{BLRH}
together with \cite[Prop.~11.31]{BSWitt}. 
\end{proof}

\begin{proposition} 
\label{compactoverwfr}
For any perfect ring $R$, the category of algebraic $\wfr$-modules is locally
regular coherent. 
Moreover, any holonomic $\wfr$-module $M$ has finite 
projective dimension over $\wfr$. 
Thus, the subcategory of compact objects in $\mathcal{D}_{\mathrm{alg}}(\wfr)$
is precisely the subcategory of $t$-bounded objects whose cohomology
$\wfr$-modules are
holonomic. 
\end{proposition} 
\begin{proof} 
Given a holonomic $\wfr$-module
$M$, there exists a perfectly finitely generated  subring $R_0 \subset R$, a
holonomic $W(R_0)[F^{\pm 1}]$-module $M_0$, and an isomorphism $M \simeq W(R)[F^{\pm 1}]
\otimes_{W(R_0)[F^{\pm 1}]} M_0$, thanks to \Cref{holonomicgenerators}. By \Cref{texactbasechange}, the isomorphism also holds with the derived
tensor product. It follows 
from \Cref{lem:noetherian}
that $M$ has finite projective dimension as a
$\wfr$-module.

Similarly, given any map $f: M \to N$ between holonomic $\wfr$-modules, we can
descend the map to some $R_0'$ and deduce (using \Cref{lem:noetherian} and
$t$-exactness of base-change again) that the kernel of $f$ is holonomic. 

The
objects of \Cref{algebraicgenerators}, which generate the category of algebraic
$\wfr$-modules, have projective dimension $\leq 2$ as $\wfr$-modules.

Combining all these assertions, we conclude that algebraic $\wfr$-modules form a
locally regular coherent Grothendieck abelian category, and the last claim
	follows from \Cref{locregcritabelian}. 
\end{proof}

\section{$\wfr$-modules and sheaves}
Let $R$ be a perfect $\mathbb{F}_p$-algebra. 
The purpose of this section is to embed $\mathcal{D}(\wfr)_{\ptors}$ as a full subcategory of
$\mathcal{D}( \schp_R)$; this embedding is a consequence of the results of
\cite{Breen}. 

Let us first recall some aspects of Morita theory. 
Let $\mathcal{C}$ be a presentable, stable $\infty$-category. 
Given an object $X \in \mathcal{C}$, we have an 
associative ring spectrum $\mathrm{End}_{\mathcal{C}}(X)$ and an 
adjunction
\[ ( X
\otimes_{\mathrm{End}_{\mathcal{C}}(X)}
(-), \mathrm{Hom}_{\mathcal{C}}(X,
-)): \mathrm{RMod}(\mathrm{End}_{\mathcal{C}}(X)) \rightleftarrows \mathcal{C}
.\]
This can be constructed as follows. 
The $\mathrm{Ind}$-completion of the thick subcategory $\mathcal{C}_0$
of $\mathcal{C}$ generated by $X$ is equivalent to 
$\mathrm{RMod}(\mathrm{End}_{\mathcal{C}}(X))$ by the Schwede--Shipley theorem
\cite[Th.~7.1.2.1]{HA}; now the inclusion $\mathcal{C}_0 \subset \mathcal{C}$ extends to a
cocontinuous functor $\mathrm{Ind}(\mathcal{C}_0) \to \mathcal{C}$ and gives the
desired adjunction. Compare also \cite[Ex.~C.1.5.11]{SAG} for an account of a  very similar
construction.

\begin{construction}[From Frobenius modules to sheaves] 
\label{Frobtosheaf}
We define a cocontinuous functor
\[ W \otimes^{\mathbb{L}}_{\wfr} (-):  \mathcal{D}(\wfr)  \to \mathcal{D}(
\schpp_R) \]
as follows. 
The Witt vector functor  defines an object $W$ of $\mathcal{D}(\schp_R)$ with a
right action of $\wfr$ (where $F$ acts as the \emph{inverse} of the Witt vector
Frobenius, cf.~\Cref{relativetensor}), whence Morita theory
produces a cocontinuous functor $\mathcal{D}(W(R)[F^{\pm 1}]) \to
\mathcal{D}(\schpp_R)$ whose the right adjoint is given by
$\mathrm{RHom}_{\mathcal{D}(\schpp_R)}( W, -)$.   
\end{construction} 

\begin{remark}  \label{presheafvssheaf}
A priori, the tensor product of \Cref{Frobtosheaf} involves \'etale sheafification, but in
fact the construction can be carried out purely at the \emph{presheaf} level for
$p$-power torsion objects:
that is,
for any $M \in \mathcal{D}(\wfr)_{\ptors}$,
the construction carrying a  perfect $R$-algebra $R'$ to $W(R')
\otimes^{\mathbb{L}}_{\wfr} M \in \mathcal{D}(\mathbb{Z})$  is already a
hypercomplete \'etale sheaf.

In fact, let $R'^{-1}$ be a perfect $R$-algebra, and let $R'^{\bullet}$ be an \'etale
hypercover of $R'^{-1}$. We need to show that the map 
$ W(R'^{-1}) \otimes^{\mathbb{L}}_{\wfr} M \to \varprojlim W(R'^\bullet)
\otimes^{\mathbb{L}}_{\wfr} M$ is an equivalence. 
Without loss of generality, we may base change and thus assume $R'^{-1} = R$. 

By working up the Postnikov
tower of $M$, we may assume that $M$ is truncated and then (by flatness of each
term in $R'^{\bullet}$ over $R$) that
$M$ is discrete, and even annihilated by $p$. 
In this case, $W(R') \otimes^{\mathbb{L}}_{\wfr} M$ is
expressed as the cofiber of the self-map $F-1$ on $R' \otimes^{\mathbb{L}}_R M
$, cf.~\Cref{relativetensor}. 
Since 
$R' \mapsto R' \otimes^{\mathbb{L}}_R M $ satisfies flat hyperdescent
\cite[Cor.~D.6.3.3]{SAG}, we conclude. 
\end{remark}

\begin{theorem} 
\label{Breenthm}
The functor 
$W \otimes_{W(R)[F^{\pm 1}]}^{\mathbb{L}}(-): \mathcal{D}(\wfr)_{\ptors} \to
\mathcal{D}(\schp_R)$ of \Cref{Frobtosheaf} is fully faithful, with right
adjoint (and left inverse)  
$\rhom_{\mathcal{D}(\schp_R)}(W, -)$.
The essential image is a localizing subcategory.\footnote{In the next section, we
will modify this by a shift, but this does not affect the conclusions.} 
\end{theorem} 
\begin{proof} 
We need to show that if $M, M' \in \mathcal{D}(\wfr)$ and if $M'$ is $p$-power
torsion, then the comparison map
\begin{equation} \label{Rhomcomp}  \mathrm{RHom}_{\wfr}(M, M') \to 
\mathrm{RHom}_{\mathcal{D}( \schp_R)}(W 
\otimes^{\mathbb{L}}_{\wfr} M, W 
\otimes^{\mathbb{L}}_{\wfr} M'
) \end{equation} 
is an equivalence. 
We may assume for the rest of
the proof that $M = \wfr$; indeed, given $M'$, the collection of $M$ such that
the comparison map is an equivalence is a localizing subcategory of $\mathcal{D}(\wfr)$ and
$\mathcal{D}(\wfr)$ is generated by $\wfr$ as a localizing subcategory. 
Thus, we want the map 
\begin{equation} \label{RHomcomp2}  M' \to \mathrm{RHom}_{\mathcal{D}( \schpp_R
)}(W,
W \otimes^{\mathbb{L}}_{\wfr} M')  
\end{equation}
to be an equivalence for any $M' \in \mathcal{D}(\wfr)_{\ptors}$.

Let $(f^*, f_*)$ be the geometric morphism between the perfect site of
$\mathbb{F}_p$ and the perfect site of $R$. 
Then $W \in 
\mathcal{D}( \schpp_R)$ is $f^*W$ by representability, 
and $f_* (W \otimes^{\mathbb{L}}_{\wfr} M')   =W
\otimes^{\mathbb{L}}_{W(\mathbb{F}_p)[F^{\pm
1}]} M'$ (thanks to \Cref{presheafvssheaf})
whence we reduce
showing
\eqref{RHomcomp2} is an equivalence to the absolute case $R = \mathbb{F}_p$,
which we do for the remainder of the proof. 

Using \Cref{presheafvssheaf}, we can further reduce to the case where $M'$ is
truncated by passage up the Postnikov tower of $M'$. 
Since $W$ is pseudo-coherent (i.e., compact in any truncation) in 
$\mathcal{D}( \schpp_{\mathbb{F}_p})$ by the Breen--Deligne resolution
(since $W$ is representable,
cf.~\cite[Lec.~4]{condensed}), we can reduce further to the case where $M'$ is
discrete and $pM' = 0$. 

Now when $M' = \mathbb{F}_p[F^{\pm 1}]$ as well, the equivalence \eqref{RHomcomp2} is
precisely the vanishing result of Breen, \cite{Breen} (also recalled below as
\Cref{Breenthmappendix}), which gives
$\mathrm{RHom}_{\mathcal{D}(\schp_{\mathbb{F}_p}, \mathbb{F}_p)}(\mathbb{G}_a, \mathbb{G}_a)
= \mathbb{F}_p[F^{\pm 1}]$. 
By pseudo-coherence again, the equivalence follows for any free (discrete) 
$\wfr/p = \mathbb{F}_p[F^{\pm 1}]$-module. 
Since $\wfr/p= \mathbb{F}_p[F^{\pm 1}]$ has global dimension $1$, we deduce
\eqref{RHomcomp2} 
for any discrete $\wfr/p$-module, and conclude. 
\end{proof}

\section{The covariant Riemann--Hilbert correspondence}

Let $R$ be a perfect $\mathbb{F}_p$-algebra. 
In this section, we construct the $\rhcov$ functor 
from $\mathcal{D}(\spec(R)_{\mathrm{et}})_{\ptors}$ to $\mathcal{D}(\wfr)$ and
prove \Cref{mainthm} below. 
We start by reviewing a version of the Artin--Schreier sequence. 

\begin{remark}[The Artin--Schreier sequence for animated
$\mathbb{F}_p$-algebras] 
\label{ASanimatedalg}
Let $A$ be any animated $\mathbb{F}_p$-algebra. 
Then there is a natural fiber sequence
\[ R \Gamma_{\mathrm{et}}( \spec( A), \mathbb{F}_p) \to A \xrightarrow{F-1} A   \]
where the first term is also 
$R \Gamma_{\mathrm{et}}( \spec( \pi_0 A), \mathbb{F}_p)$. This follows from the
usual Artin--Schreier sequence for classical $\mathbb{F}_p$-algebras and the
fact that
Frobenius acts by zero on the
higher homotopy groups of an animated $\mathbb{F}_p$-algebra, \cite[Prop.~11.6
and proof]{BSWitt}. 
\end{remark}

\begin{proposition} 
\label{FhasimageinRH}
Let $\mathcal{F} \in \mathcal{D}(\spec(R)_{\mathrm{et}})_{\ptors}$. Then 
there exists a (unique) $M \in \mathcal{D}(\wfr)_{\ptors}$ and an equivalence
$\pi^*
\mathcal{F} \simeq W \otimes^{\mathbb{L}}_{\wfr}M[-1]$. 
Moreover, $M$ is algebraic. 
\end{proposition} 
\begin{proof} 
By \Cref{Breenthm}, the functor $W \otimes^{\mathbb{L}}_{\wfr} (-)[-1]$ is fully
faithful with image a localizing subcategory. 
It suffices to prove the proposition when $\mathcal{F} = j_! \mathbb{F}_p$, for $j: U \to
\spec(R)$ an affine, \'etale map, since these generate 
$\mathcal{D}(\spec(R)_{\mathrm{et}})_{\ptors}$ as a localizing subcategory. 
Using Zariski's main theorem
\cite[Tag 05K0]{stacks-project}, we can factor $j$ as the composite of an open
immersion $U \to Z$ and a finite morphism $g: Z \to \spec(R)$. 
Thus, $j_!(\mathbb{F}_p)$ is the  kernel (or fiber) of the (surjective) map from $g_*(\mathbb{F}_p)$ to
its restriction to a constructible closed subset.
It thus suffices to prove the result for $\mathcal{F}$ the pushforward of the
constant sheaf along any finite map to $\spec(R)$. 

Let $f: \spec B \to \spec R$ be a finite map. 
Then the object $\pi^* (f_* \mathbb{F}_p)$ carries a perfect $R$-algebra $R'$ to the
mod $p$ \'etale cohomology 
of $R' \otimes_R B $, or equivalently $R' \otimes^{\mathbb{L}}_R
B_{\mathrm{perf}}$. 
Thanks to the Artin--Schreier sequence (in the form of \Cref{ASanimatedalg}), $\pi^* ( f_* \mathbb{F}_p)$ 
can be described as the functor which carries
a  perfect $R$-algebra $R'$ to the Frobenius fixed points of $R'
\otimes^{\mathbb{L}}_R
B_{\mathrm{perf}}$. In particular, it corresponds (cf.~\eqref{cofibexpr}) to the 
$\wfr$ module $B_{\mathrm{perf}}$, which is also algebraic
by \Cref{integralisalgebraic}. 
\end{proof} 
\begin{construction}[The functor $\rhcov$] 
\label{consRH}
By 
\Cref{FhasimageinRH}, 
the fully faithful left adjoint functor $\pi^* :  \mathcal{D}( \spec(R)_{\mathrm{et}}
)_{\ptors} \to  
\mathcal{D}( \schp_R)$ has essential image inside the image of the
fully faithful left adjoint functor $W \otimes^{\mathbb{L}}_{\wfr} (-)[-1]$
(cf.~\Cref{Breenthm}), so $\pi^* $
uniquely factors through a fully faithful left adjoint functor
$\rhcov: \mathcal{D}( \spec(R)_{\mathrm{et}})_{\ptors} \to \mathcal{D}( \wfr
)_{\ptors}$,
i.e., we have a commutative diagram 
of fully faithful, left adjoint functors,
\[ \xymatrix{
	& \mathcal{D}(\wfr)_{\ptors} \ar[d]^{W  \otimes^{\mathbb{L}}_{\wfr}
 (-)[-1]}  \\
\mathcal{D}( \spec(R)_{\mathrm{et}})_{\ptors} \ar[r]^{\pi^*} \ar[ru]^{\rhcov} &
\mathcal{D}(\schp_R). 
}.\]
Explicitly, for any $\mathcal{F} \in
\mathcal{D}(\spec(R)_{\mathrm{et}})_{\ptors}$, we
have by \Cref{Breenthm},
\begin{equation} \label{RHformula} \rhcov(\mathcal{F}) =
\mathrm{RHom}_{\mathcal{D}(\schp_R)}(W, \pi^*
\mathcal{F}[1]). \end{equation} 
\label{RHcons}
\end{construction}

\begin{remark}[Comparison with the $\rhcov$ of \cite{BLRH}] 
Note that $\rhcov$ is a left adjoint, and by the above diagram of left adjoints
(with the right vertical arrow fully faithful), we find that $\rhcov$ is
left adjoint to the
functor (referred to as the \emph{solution} functor in \cite{BLRH}) which
carries $M \in \mathcal{D}(\wfr)_{\ptors}$ to 
$\pi_*( W \otimes^{\mathbb{L}}_{\wfr} M) [-1] \in
\mathcal{D}(\spec(R)_{\mathrm{et}})_{\ptors}$. 
 In particular, 
the construction of $\rhcov$ here (in light also of the expression
\eqref{cofibexpr} for the tensor product, and \Cref{presheafvssheaf}) agrees with that of
\cite[Th.~6.1.1]{BLRH}.\end{remark} 

Let us now formulate the main result, i.e., the $p$-power torsion version of the
Riemann--Hilbert correspondence for $\mathbb{F}_p$-schemes,
cf.~\cite[Th.~9.6.1]{BLRH} for the mod $p^n$-version (and on abelian
categories); the present formulation of the result is not difficult to deduce from there. 

\begin{theorem} 
\label{mainthm}
The functor  $\rhcov: \mathcal{D}(\spec(R)_{\mathrm{et}})_{\ptors} \to
\mathcal{D}(\wfr)_{\ptors}$ is fully faithful, $t$-exact, and cocontinuous. The image consists of those $M \in \mathcal{D}(\wfr)$
which are algebraic. \end{theorem}

By the full faithfulness of $\pi^*$, the essential image of $\rhcov$ consists of 
precisely those objects $M \in \mathcal{D}(\wfr)_{\ptors}$ 
such that $W \otimes^{\mathbb{L}}_{\wfr} M$ belongs to the 
essential image of $\pi^*$. 
We have already seen (in \Cref{FhasimageinRH}) that any such $M$ is algebraic, so it remains to prove the
converse. 

\begin{proposition} 
\label{hensrigidity}
Let $M \in \mathcal{D}(\wfr)$ be algebraic. Let $(R', I)$ be a henselian pair of
perfect $R$-algebras (so $R'/I$ is also assumed perfect). Then 
the map 
$W(R') \otimes_{\wfr}^{\mathbb{L}} M 
\to W(R'/I) 
\otimes_{\wfr}^{\mathbb{L}} M
$ is an equivalence. 
\end{proposition} 
\begin{proof} 
In view of the generators of the localizing subcategory of algebraic objects of
$\mathcal{D}(\wfr)$ of \Cref{algebraicgenerators}, we see that
it suffices to show that any monic element $F^n + a_1 F^{n-1} + \dots + a_n$
acts invertibly on $I= \mathrm{ker}(W(R') \to W(R'/I))/p$ by right multiplication. 
In other words, it suffices to show that the self-map of $I$ given by 
\begin{equation} \label{selfmap} x \mapsto x^{1/p^n} + a_1^{1/p^{n-1}} x^{1/p^{n-1}} + \dots
+  a_n x  \end{equation}
is an isomorphism. 
Since $I$ is perfect as a nonunital ring, it suffices to precompose by the $n$th
iterate of Frobenius and to
show that the map 
\begin{equation}\label{selfmap2}  x \mapsto x  + a_1^{1/p^{n-1}} x^p + \dots + a_n x^{p^n}, \quad I \to I
\end{equation}
is an isomorphism. 
However, this follows from the definition of a henselian pair. 
In fact, for any $t \in I$, the polynomial $x  + a_1^{1/p^{n-1}} x^p + \dots + a_n x^{p^n} -
t\in R'[x]$ has
nonvanishing derivative, so any solution in $R'/I$ (e.g., zero) lifts uniquely to
a solution in $I$, whence the claim. 
\end{proof}

\begin{proposition} 
\label{algebraicinessimage}
Let $M \in \mathcal{D}(\wfr)$ be algebraic. Then 
$W \otimes^{\mathbb{L}}_{\wfr} M[-1] \in \mathcal{D}(\schp_R)$ belongs to the
essential image of $\pi^*: \mathcal{D}(\spec(R)_{\mathrm{et}})_{\ptors} \to
\mathcal{D}( \schpp_R)$. \end{proposition} 
\begin{proof} 
Thanks to \Cref{presheafvssheaf}, the construction 
$W \otimes^{\mathbb{L}}_{\wfr} M[-1]$ as a functor on perfect $R$-algebras
commutes with filtered colimits. 
Let $R_1 \to R_2$ be a local homomorphism of strictly henselian, perfect
$R$-algebras. We show that the map 
\[ W(R_1) \otimes^{\mathbb{L}}_{\wfr} M \to W(R_2) \otimes^{\mathbb{L}}_{\wfr} M  \]
is an equivalence, which will suffice, cf.~\Cref{imageofpi}. Thanks to
\Cref{hensrigidity}, it suffices to show this when
$R_1, R_2$ are algebraically closed fields. 
Since the algebraic subcategory of $\mathcal{D}(\wfr)$ is generated as a
localizing subcategory
(cf.~\Cref{algebraicgenerators}) by
objects of the form $\wfr/(\wfr T, p)$ for $T \in \wfr$ of
the form $F^n + a_1 F^{n-1} + \dots + a_n$, it suffices to assume $M$ to be of
this form. 

In this case, 
$W(R_1) \otimes^{\mathbb{L}}_{\wfr} M$ is given by 
the cofiber of the map 
\eqref{selfmap} on $R_1$. This is clearly zero in $H^0$ and a finite set of
cardinality $\leq p^n$ in
$H^{-1}$ (the set of solutions of the equation $x^{1/p^n}+ a_1^{1/p^{n-1}}x^{1/p^{n-1}} + \dots + a_n x =0
$), which is unchanged under extensions of algebraically closed fields. 
\end{proof} 

\begin{proof}[Proof of \Cref{mainthm}] 
We have already seen the fully faithful, cocontinuous functor $\rhcov$ in 
\Cref{RHcons}. The essential image was seen 
in \Cref{FhasimageinRH} to be contained inside the algebraic subcategory of 
$\mathcal{D}(\wfr)$, but \Cref{algebraicinessimage} now shows that the
essential image  is precisely the algebraic subcategory. 

It only remains to check $t$-exactness, i.e., that 
if
an algebraic object
$M \in \mathcal{D}(\wfr)$ is connective (resp.~coconnective), then 
$ W \otimes^{\mathbb{L}}_{\wfr} M[-1] \in \mathcal{D}(
\schpp_R)$ is connective (resp.~coconnective). Here we also use
that $\pi^*$ reflects connectivity (resp.~coconnectivity). We carry this out in
the following two paragraphs, and implicitly use \Cref{presheafvssheaf}. 

If $M$ is coconnective and $M[1/p] = 0$, then 
for any \emph{flat}, perfect $R$-algebra $R'$, we have that $W(R') \otimes^{\mathbb{L}}_{\wfr}
M [-1]$ is coconnective.
In fact, this reduces by taking colimits and extensions to the case where $M$ is discrete and annihilated by $p$,
whence the result follows from the expression \eqref{cofibexpr}. 
We thus obtain the same coconnectivity claim for any perfect $R$-algebra $R'$,
since we can replace (thanks to \Cref{hensrigidity}) $R'$ with the henselization
(along the kernel) of the
perfection of a polynomial $R$-algebra surjecting onto $R'$, which is flat over
$R$. 

If $M \in \mathcal{D}(\wfr)$ is connective and algebraic, 
then for any strictly henselian, perfect $R$-algebra $R'$, 
we claim that $W(R') \otimes^{\mathbb{L}}_{\wfr} M[-1]$ is connective.
By  \Cref{hensrigidity}, we may assume that $R' = k$ is actually an algebraically
closed field $k$. 
In light of \Cref{algebraicgenerators} (and closure under colimits
and extensions), we may
assume that $M = W(R) [F^{\pm 1}]/(p, T )$ for 
$T = F^n + a_1 F^{n-1} + \dots + a_n, a_1, \dots, a_n \in W(R)$. 
But right multiplication by $T$, i.e. the map 
\eqref{selfmap}
is clearly surjective on $k$; in fact, this reduces to showing that 
\eqref{selfmap2} is surjective on $k$, which follows because $k$ is
algebraically closed. \end{proof}

\begin{proof}[Proof of \Cref{BLthm}] 
Note that \Cref{BLthm} is stated for an arbitrary $\mathbb{F}_p$-algebra $R$, not
necessarily assumed to be perfect; however, the statement reduces to the perfect
case since passage to the perfection does not change the \'etale site. 
This reduction follows because the multiplicative subset $\left\{F^{i}\right\}_{i \geq 0}
\subset R[F]$ satisfies the left Ore condition, and the localization is given by
$R_{\mathrm{perf}}[F^{\pm 1}]$. The restriction functor on derived
$\infty$-categories $\mathcal{D}(R_{\mathrm{perf}}[F^{\pm 1}]) \to
\mathcal{D}(R[F])$ is fully faithful, cf.~\cite[Sec.~7.2.3]{HA}. 
Hence, we assume that $R$ is perfect, whence the result follows by passing to
$\mathbb{F}_p$-modules in \Cref{mainthm}. 
\end{proof}

\begin{remark} 
The treatment in \cite{BLRH, BocklePink} first proves the
correspondence at the level of abelian categories.  
Here, it is essential to work at the level of derived $\infty$-categories
throughout, for the use of Morita theory. Of course, in the
contravariant approach \cite{EmertonKisin}, the use of derived
$\infty$-categories is essential as the functors involved are not
$t$-exact. \end{remark} 
\begin{remark} 
The equivalence of \Cref{mainthm} yields an equivalence on compact objects: one
obtains a $t$-exact equivalence between the bounded constructible
$p$-power torsion derived
$\infty$-category of $\spec(R)_{\mathrm{et}}$ and 
the bounded  derived $\infty$-category of $\wfr$ consisting of objects with holonomic
cohomologies, by 
\Cref{compactobjofderivedcat} and \Cref{compactoverwfr}. 
\end{remark}

\begin{remark}[Compatibilities] 
As proved in \cite{BLRH}, \Cref{BLthm} is compatible with pullback, proper pushforward, and 
tensor products. The first two claims can also be proved analogously using the
perfect site: the functor from Frobenius modules to sheaves on the perfect site
is compatible with pullback and (when globalized as in \cite[Sec.~10]{BLRH}) with proper pushforward. 
However, the compatibility with the symmetric monoidal structure appears less clear
from this perspective; for example, the description of $\rhcov(\mathcal{F}) =
\mathrm{RHom}_{\mathcal{D}(\schp_R)}(W, \pi^*
\mathcal{F}[1])$ does not have an evident symmetric monoidal structure. 
\end{remark} 

\section{The contravariant Riemann--Hilbert correspondence}

Let $R$ be a  regular  noetherian $\mathbb{F}_p$-algebra. 
In this case, the work of Emerton--Kisin \cite{EmertonKisin} gives a contravariant description of the
bounded derived $\infty$-category of constructible $\mathbb{F}_p$-sheaves,
$\mathcal{D}^b_{\mathrm{cons}}( \spec(R)_{\mathrm{et}}, \mathbb{F}_p)$, in terms
of finitely generated unit Frobenius modules.\footnote{The work
\cite{EmertonKisin} also treats the case of $\mathbb{Z}/p^n$-sheaves when a lift
of $R$ to $\mathbb{Z}/p^n$ is specified, and the correspondence has been extended to certain singular
cases in \cite{Ohkawa, Schedlmeier}.
An analog of this result for arbitrary $\mathbb{F}_p$-algebras can also be deduced from \Cref{BLthm} using a duality argument at the
level of Frobenius modules, cf.~\cite[Sec.~12]{BLRH}. 
We do not treat these extensions here.}

In this section, we give another
proof of the contravariant Riemann--Hilbert correspondence of
\cite{EmertonKisin}
(\Cref{EKthm}, reproduced below as \Cref{EKthm2}). The essential observation  is that
the subcategory
$\mathcal{D}^{\mathrm{fproj}}_{\mathrm{unit}}(R[F])^{op} \subset \mathcal{D}(R[F])$ 
consisting of objects of finite projective dimension whose cohomology $R[F]$-modules are unit
embeds fully faithfully into the derived $\infty$-category of sheaves
on the \emph{big} \'etale site of $\spec(R)$ (\Cref{fullfaithfulnessofSol}). 

\subsection{Unit $R[F]$-modules}
Let $R$ be an  $\mathbb{F}_p$-algebra whose Frobenius $\phi: R \to R$ is flat. 
In this subsection, we review some basic facts about unit $R[F]$-modules.

%We have also included a result (\Cref{unitisderivedofheart}) that 
%$\mathcal{D}_{\mathrm{unit}}(R[F])$ is identified with the derived
%$\infty$-category of the abelian category 
%of the abelian category of unit $R[F]$-modules. This will not be strictly necessary for the
%proof of \Cref{EKthm2}, but implies (for example) that 
%$\mathcal{D}_{\mathrm{unit}}(R[F])$
%is identified with the derived $\infty$-category of a certain large
%noncommutative ring in the case when $R$ is $F$-finite regular noetherian. 

We recall that an $R[F]$-module can be equivalently regarded as an $R$-module $M$ equipped with an $R$-linear map $\phi^* M
\to M$ (adjoint to $F: M \to \phi_* M$); we will abuse notation and often write
this map as $F$ too.

\begin{definition}[{Unit $R[F]$-modules}] 
An $R[F]$-module $M$ is \emph{unit} if the map $F: \phi^* M \to M$ of $R$-modules
is an isomorphism. 

Since $\phi^*$ is exact by assumption, the collection of unit $R[F]$-modules is
closed under all colimits, finite limits, and extensions inside the category of all
$R[F]$-modules; in particular, it is an abelian subcategory. 
\end{definition}
When $R$ is regular noetherian,\footnote{Recall that a noetherian
$\mathbb{F}_p$-algebra is regular if and only if its Frobenius is flat,
\cite{Kunz1}.} 
the abelian category of unit $R[F]$-modules is extensively studied by Lyubeznik,
\cite{Lyubeznik}, under the name \emph{$F$-modules}.

\begin{remark}[Unitality via the perfection] 
An $R[F]$-module $M$ is unit if and only if the base-change $R_{\mathrm{perf}}
\otimes_R M$, considered as an $R_{\mathrm{perf}}[F]$-module, has $F$ acting
invertibly. This follows because the condition that $\phi^* M \to M$ is an
isomorphism can be checked after faithfully flat
base-change, and for a perfect $\mathbb{F}_p$-algebra is equivalent to $F$
acting invertibly. 
\end{remark} 
\begin{remark}[{Unit $R[F]$-modules as modules over a ring}]
\label{unitasmodules}
Suppose $R$ is $F$-finite regular noetherian. 
The abelian category of unit $R[F]$-modules admits a compact projective
generator. In fact, this follows because the forgetful functor from unit
$R[F]$-modules to $R$-modules preserves limits and colimits (because $\phi^*$
preserves limits by the $F$-finiteness assumption; preservation of colimits
always holds), and so admits a
left adjoint; the image of $R$ under this left adjoint is a compact projective
generator. Consequently, the category of unit $R[F]$-modules is identified with
the category of modules over a certain large ring (containing $R[F]$). 
A description of this ring (at least when $R$ is smooth over a perfect field) in terms of differential operators on $R$ appears in \cite[Prop.~15.1.4]{EmertonKisin} (see
also \cite[Sec.~15.2]{EmertonKisin}). 
\end{remark}

\begin{construction}[{Unitalization, cf.~\cite[Def.~1.9]{Lyubeznik} or
\cite[Cons.~11.2.2]{BLRH}}] 
\label{unitalization}
Let $M$ be an $R$-module together with a map $f: M \to \phi^* M$ of $R$-modules. The colimit $$M
\stackrel{f}{\to} \phi^* M  \stackrel{\phi^* f}{\to } (\phi^{*})^2 M \to \dots$$
naturally has the structure of a unit $R[F]$-module $M^{\mathrm{unit}}$: in
fact, the construction provides an identification of $M^{\mathrm{unit}}$ and
$\phi^* M^{\mathrm{unit}}$. 
More generally, one can carry out this construction when $M \in \mathcal{D}(R)$. 
%In fact, this is a special case of \Cref{coalgforendofunctor},
%applied to the endofunctor $\phi^*$. 
\end{construction} 

\begin{proposition} 
\label{univpropertyunit}
Let $M $ be an $R$-module equipped with a map $f: M \to \phi^* M$. 
Then the unitalization $M^{\mathrm{unit}}$ of $(M, f)$ has the following universal mapping property: 
given $N \in \mathcal{D}(R[F])$, there is an equivalence in
$\mathcal{D}(\mathbb{Z})$, 
$$\rhom_{R[F]}(M^{\mathrm{unit}}, N) =  
\mathrm{eq} \left( \rhom_R(M, N) \rightrightarrows \rhom_R(M, N) \right)
$$
where the first map is the identity and the second map carries $v: M \to N$ to 
the composite
$M \stackrel{f}{\to} \phi^* M \stackrel{\phi^* v}{\to} \phi^* N
\stackrel{F}{\to}N$. 
\end{proposition} 
\begin{proof} 
This follows from the presentation of the unitalization given in
\cite[Prop.~11.2.5]{BLRH} upon taking maps in the derived $\infty$-category. 
\end{proof}

\begin{definition}[{Unit objects of $\mathcal{D}(R[F])$}]
An object of $\mathcal{D}(R[F])$ is said to be \emph{unit} if all the cohomology
$R[F]$-modules are unit. We let $\mathcal{D}_{\mathrm{unit}}(R[F]) \subset
\mathcal{D}(R[F])$ denote the subcategory of unit objects. \end{definition} 

The subcategory $\mathcal{D}_{\mathrm{unit}}(R[F]) \subset \mathcal{D}(R[F])$ is
closed under colimits and finite limits, since unit $R[F]$-modules are closed
under all colimits, finite limits, and extensions inside all $R[F]$-modules. 
Moreover, 
$\mathcal{D}_{\mathrm{unit}}(R[F])$ 
is accessible by general accessibility results 
(e.g., the stability of accessible $\infty$-categories under limits,
\cite[Prop.~5.4.7.3]{HTT}) and is therefore presentable stable. 

Note that 
$\mathcal{D}_{\mathrm{unit}}(R[F])$ inherits a natural $t$-structure from the
inclusion into $\mathcal{D}(R[F])$.

\begin{proposition} 
\label{unitisderivedofheart}
$\mathcal{D}_{\mathrm{unit}}(R[F])^{\leq 0}$ is Grothendieck prestable.
Moreover, 
$\mathcal{D}_{\mathrm{unit}}(R[F])$ is identified with the 
derived $\infty$-category of its heart (i.e., the abelian category of unit
$R[F]$-modules). 
\end{proposition} 
\begin{proof} 
Since the $t$-structure on 
$\mathcal{D}_{\mathrm{unit}}(R[F])$ induced from
$\mathcal{D}(R[F])$ is evidently right-complete
and compatible with filtered colimits, 
we find that
$\mathcal{D}_{\mathrm{unit}}(R[F])^{\leq 0}$ is Grothendieck prestable. 

It is clear that 
$\mathcal{D}_{\mathrm{unit}}(R[F])^{\leq 0}$ is left separated, so it suffices to
show that it is 0-complicial in light of \cite[Prop.~C.5.4.5]{SAG}. 
Let $M \in 
\mathcal{D}_{\mathrm{unit}}(R[F])^{\leq 0}$. 
We can choose a free $R$-module $P$ together with a map $g: P \to M$ in
$\mathcal{D}(R)$ which induces a surjection on $H^0$. 
Moreover, since $g$ is surjective on $H^0$ and since $P$ is free, we can choose a map 
of $R$-modules $f: P \to \phi^* P$ such that the composite 
$P \stackrel{f}{\to} \phi^* P \stackrel{\phi^* g}{\to} \phi^* M
\stackrel{F}{\simeq} M$ agrees
with $g$. Using the universal property of the unitalization, 
\Cref{univpropertyunit}, we obtain a map 
in $\mathcal{D}(R[F])$ from the unitalization of $(P, f: P \to \phi^* P)$ to $M$
which necessarily induces a surjection on $H^0$. This proves 0-compliciality and
thus the result. 
\end{proof} 

\begin{remark} 
If $R$ is $F$-finite regular noetherian, 
it follows that $\mathcal{D}_{\mathrm{unit}}(R[F])$ is simply the derived
$\infty$-category of the ring of \Cref{unitasmodules}. 
\end{remark}

As a consequence, we obtain the following result on the category of unit
$R[F]$-modules; for regular noetherian $\mathbb{F}_p$-algebras satisfying
$F$-finiteness assumptions, this result is due to Ma \cite[Th.~1.3]{Ma14} (who
also proves the lower bound).

\begin{corollary} 
Let $R$ be an $\mathbb{F}_p$-algebra whose Frobenius is flat and which has global
dimension $d$ (e.g., a regular noetherian $\mathbb{F}_p$-algebra of Krull
dimension $d$). 
Then the abelian category of unit $R[F]$-modules has global dimension $\leq
d+1$. 
\end{corollary} 
\begin{proof} 
In light of 
\cite[Rem.~3.1.8]{BLRH},  the abelian category of $R[F]$-modules has
global dimension $\leq d + 1$. 
Since $\mathrm{Ext}$-groups in unit $R[F]$-modules are computed in all
$R[F]$-modules thanks to \Cref{unitisderivedofheart}, the result follows. 
\end{proof} 

In the remainder of the subsection, we assume that $R$ is a regular noetherian
$\mathbb{F}_p$-algebra. 

\begin{definition}[{Finitely generated unit $R[F]$-modules}]
An $R[F]$-module $M$ is \emph{finitely generated unit} if it is unit and if it
is finitely generated as an $R[F]$-module. 
We let $\mathcal{D}^b_{\mathrm{fgu}}(R[F]) \subset \mathcal{D}(R[F])$ denote the
subcategory of objects which are $t$-bounded and whose cohomology $R[F]$-modules
are finitely generated unit. 
\end{definition}

\begin{proposition} 
\label{fguproperties}
\begin{enumerate}
\item (Cf.~\cite[Th.~6.1.3]{EmertonKisin} or \cite[Cor.~11.2.6, 11.2.12]{BLRH}) An $R[F]$-module $M$ is finitely generated unit if and only if there exists a
finitely generated 
$R$-module $N$ with a map $f: N \to \phi^* N$ such that $M$ arises as the
unitalization of $(N, f)$. 
\item (Cf.~\cite[Th.~2.8]{Lyubeznik}) The collection of finitely generated unit $R[F]$-modules is closed (inside
unit $R[F]$-modules) under subobjects, quotients, and extensions. 
\end{enumerate}
\end{proposition} 

It follows that 
$\mathcal{D}^b_{\mathrm{fgu}}(R[F]) \subset \mathcal{D}(R[F])$ is a thick
subcategory. 
\begin{corollary}[{Cf.~\cite[Lem.~11.3.12]{BLRH}}] 
\label{fgumeansfproj}
Any object of $\mathcal{D}^b_{\mathrm{fgu}}(R[F])$ has finite projective dimension
over $R[F]$. 
\end{corollary}

\subsection{The Emerton--Kisin correspondence}
Let $R$ be an $\mathbb{F}_p$-algebra. 
We consider the solutions functor, as in \cite{EmertonKisin}; however, we
do so on the \emph{big} \'etale site rather than the small \'etale site. 
The basic observation is that doing so leads to full faithfulness on a larger
subcategory (\Cref{fullfaithfulnessofSol}). 

\begin{definition}[The big \'etale site] 
\label{bigetale}
We let $\sch_R$ denote the site of all qcqs $R$-schemes, equipped 
with the \'etale topology, and let $\mathcal{D}( \sch_R, \mathbb{F}_p)$ denote
the derived $\infty$-category of sheaves of $\mathbb{F}_p$-vector spaces on
$\sch_R$. 

Let $(\lambda^*, \lambda_*)$ denote the pullback and pushforward from the small \'etale site to the big
\'etale site; note that $\lambda^*$ is fully faithful, and the essential image
of $\lambda^*$ 
consists of those objects of $\mathcal{D}( \sch_R, \mathbb{F}_p)$ such that on
$R$-algebras, they commute with filtered colimits and
carry local homomorphisms of strictly henselian local $R$-algebras to
equivalences
(analogously to \Cref{restrictedadjunction} and \Cref{imageofpi}).

\end{definition} 
\begin{construction}[The solutions functor] 
We have a functor
\[ \mathrm{Sol} = \mathrm{RHom}_{R[F]}(-, \mathbb{G}_a): \mathcal{D}(R[F])^{op} \to \mathcal{D}( \sch_R, \mathbb{F}_p)
\]
which carries $M \in \mathcal{D}( R[F])$ to the functor 
that sends an $R$-algebra $R'$ to $\mathrm{RHom}_{R[F]}( M, R') \in
\mathcal{D}(\mathbb{F}_p)$.  
This defines a hypercomplete \'etale sheaf with values in $\mathcal{D}(\mathbb{F}_p)$ on
the category of $R$-algebras (since $R' \mapsto R'$ does), whence an object of 
$\mathcal{D}( \sch_R, \mathbb{F}_p)$. 
\end{construction} 

\begin{remark}[Compatibility of $\mathrm{Sol} $ and base-change]
The solutions functor $\mathrm{Sol}$ carries a free $R[F]$-module $\bigoplus_I
R[F]$ to the direct product $\prod_I \mathbb{G}_a$. 
Using this, one concludes that $\mathrm{Sol}$ is compatible with base-change
\emph{on
the subcategory of objects with bounded projective amplitude}: if 
$M \in \mathcal{D}(R[F])$ has bounded projective amplitude, and
if
$f: \spec(R') \to \spec(R)$ is any map, then $f^* ( \mathrm{Sol}(M)) \simeq
\mathrm{Sol}(R' \otimes^{\mathbb{L}}_R M)$ for $f^*$ the pullback on sheaves on
the big \'etale site; indeed, this follows because $f^*$ carries the
representable $\prod_I \mathbb{G}_a$ to $\prod_I \mathbb{G}_a$. 
\label{solandbasechange}
\end{remark} 
In the following, we let 
$\mathcal{D}_{\mathrm{unit}}^{\mathrm{fproj}}(R[F]) \subset \mathcal{D}(R[F])$
denoted the subcategory spanned by objects which are unit and of bounded
projective amplitude. 
In the case when $R$ has finite global dimension, this is equivalent to simply
being $t$-bounded, thanks to 
\cite[Rem.~3.1.8]{BLRH}.

\begin{theorem} 
\label{fullfaithfulnessofSol}
Suppose $R$ is any $\mathbb{F}_p$-algebra whose Frobenius is flat. 
The restriction of $\mathrm{Sol}$ induces a fully faithful embedding
$\mathcal{D}^{\mathrm{fproj}}_{\mathrm{unit}}(R[F])^{op} \to \mathcal{D}(\sch_R, \mathbb{F}_p)$
with left inverse given by $\mathrm{RHom}_{\mathcal{D}(\sch_R,
\mathbb{F}_p)}(-, \mathbb{G}_a)$. 
More generally, for any $M \in
\mathcal{D}^{\mathrm{fproj}}_{\mathrm{unit}}(R[F]), N \in
\mathcal{D}(R[F])$,  the natural map
\begin{equation}  \rhom_{R[F]}(N, M) \to  \rhom_{\mathcal{D}( \sch_R, \mathbb{F}_p)}(
\mathrm{Sol}(M), \mathrm{Sol}(N)) \label{compmapsol} \end{equation}
is an equivalence. 
\end{theorem} 
\begin{proof} 
It suffices to treat the case where $N = R[F]$ itself, 
so that 
$\mathrm{Sol}(N) = \mathbb{G}_a$, since both sides carry colimits in $N$ to
limits. 

We let $R_{\mathrm{perf}}$ be the perfection of $R$, so that $R \to
R_{\mathrm{perf}}$ is
faithfully flat by our assumptions. 
Then $R_{\mathrm{perf}}\otimes_R M = R_{\mathrm{perf}}[F] \otimes_{R[F]} M \in
\mathcal{D}(R_{\mathrm{perf}}[F])$; since $M$ was assumed unit, it follows that 
$F$ acts invertibly on
$R_{\mathrm{perf}}\otimes_R M $, i.e., it 
belongs to the image of the fully faithful restriction functor
$\mathcal{D}(R_{\mathrm{perf}}[F^{\pm 1}]) \to
\mathcal{D}(R_{\mathrm{perf}}[F])$. 
The strategy is now to reduce the statement over $R$ to a statement over
$R_{\mathrm{perf}}$. 

Note that $\mathrm{Sol}(M)$ belongs to the thick subcategory of 
$\mathcal{D}(\sch_R, \mathbb{F}_p)$ generated by products $\prod_I
\mathbb{G}_a$; in fact, this follows because $M$ has finite projective
dimension as an $R[F]$-module by assumption. 
Using the Breen--Deligne resolution, one sees that 
the map 
$$\rhom_{\mathcal{D}(\sch_R, \mathbb{F}_p)}( \mathrm{Sol}(M), \mathbb{G}_a) 
\to \rhom_{\mathcal{D}( \sch_{R_{\mathrm{perf}}}, \mathbb{F}_p)}( \mathrm{Sol}(
R_{\mathrm{perf}}
\otimes_R M), \mathbb{G}_a)
$$
exhibits the target 
as the extension of scalars of the source from $R$ to $R_{\mathrm{perf}}$ (cf.~also
\Cref{solandbasechange}); note here that we regard the source as living in
$\mathcal{D}(R[F])$ and the target in $\mathcal{D}(R_{\mathrm{perf}}[F])$, because of
the presence of the $\mathbb{G}_a$'s.
In fact, by a thick subcategory argument, this reduces to the claim that for any set $I$, the map 
\begin{equation} \label{tensormap2} \rhom_{\mathcal{D}(\sch_R, \mathbb{F}_p)}( \prod_I \mathbb{G}_a,
\mathbb{G}_a) 
\to 
\rhom_{\mathcal{D}(\sch_{R_{\mathrm{perf}}}, \mathbb{F}_p)}( \prod_I \mathbb{G}_a,
\mathbb{G}_a) 
\end{equation}
exhibits the target as the extension of scalars of the source from $R$ to
$R_{\mathrm{perf}}$; this in turn follows from the Breen--Deligne resolution for $\prod_I
\mathbb{G}_a$.  Alternatively, one observes that the pushforward of
$\mathbb{G}_a$ from $\sch_{R_{\mathrm{perf}}}$ to $\sch_R$ is $\mathbb{G}_a
\otimes_R R_{\mathrm{perf}}$; 
since  $\prod_I
\mathbb{G}_a$ is pseudo-coherent and $R_{\mathrm{perf}}$ is a filtered colimit of
finitely generated free $R$-modules by Lazard's theorem, the claim about \eqref{tensormap2} follows.  

Thus, in order to prove that 
\eqref{compmapsol} is an equivalence, 
it suffices to show that the map 
$$ R_{\mathrm{perf}}\otimes_R M \to \rhom_{\mathcal{D}(
\sch_{R_{\mathrm{perf}}}, \mathbb{F}_p)}( \mathrm{Sol}(R_{\mathrm{perf}}
\otimes_R M), \mathbb{G}_a) $$ is an equivalence. 
Since $R_{\mathrm{perf}}\otimes_R M \in \mathcal{D}(R_{\mathrm{perf}}[F^{\pm 1}])$ has bounded
projective amplitude (because it has bounded projective amplitude over
$R_{\mathrm{perf}}[F]$ by assumption on $M$), the claim 
follows from the 
results of \cite{Breen} (see \Cref{variantBreenthm}). 
\end{proof} 

Before stating and proving \Cref{EKthm2} below, we need some lemmas about
\'etale sheaves and $p$-linear algebra. 
\begin{lemma} 
\label{fieldrigidity}
Let $k$ be a separably closed field of characteristic $p$. Let $V$ be a
$k[F]$-module which is finite-dimensional as a $k$-vector space. 
Then for any extension $k \subset k'$ of separably closed fields, the map 
\[ \mathrm{fib}( F-1: V \to V ) \to \mathrm{fib}(F- 1 : k' \otimes_k V \to k'
\otimes_k V)\]
is an equivalence. 
\end{lemma} 
\begin{proof} 
By \cite[Exp.~XXII, Prop.~1.2]{SGA7}, the operator $F-1$  acts surjectively on $V$ and
on $k' \otimes_k V$. 
Thus, it suffices to prove the 
assertion on $H^0$, which follows from \cite[Exp.~XXII, Cor.~1.1.10]{SGA7} (noting also
\cite[Exp.~XXII, Eq.~(1.0.10)]{SGA7}). 
\end{proof} 
\begin{lemma} 
\label{rigidity2}
Let $M$ be an $n$-by-$n$ matrix with coefficients in the
$\mathbb{F}_p$-algebra $R$. 
The construction which carries an $R$-algebra $R'$ to the mapping fiber of 
$R'^n \stackrel{1 - M \phi}{\to} R'^n$ 
carries local homomorphisms of strictly henselian local $R$-algebras to
equivalences. 
\end{lemma} 
\begin{proof} 
We have already seen that the conclusion of the lemma holds for inclusions of
separably closed fields (\Cref{fieldrigidity}). Thus, it suffices to treat the
case of the map from a strictly henselian local $R$-algebra $R'$ to 
its residue field. 

There is a map $h: \mathbb{A}^n_R \to \mathbb{A}^n_R$
such that on $R'$-points, it is given by the map  
$R'^n \stackrel{1 - M \phi}{\to} R'^n$; this map is \'etale by the Jacobian
criterion \cite[Tag 02GU]{stacks-project}.  
Now by the lifting property for \'etaleness, any diagram
\[ 
\xymatrix{
\spec(R'/\mathfrak{m}) \ar[d]  \ar[r] &  \mathbb{A}^n_R \ar[d]^{h}  \\
\spec(R') \ar[r] 
\ar@{-->}[ru] 
&  \mathbb{A}^n_R
}   \]
admits a unique 
dotted arrow
that makes the extended diagram commute \cite[Tag 08HQ]{stacks-project}. 
Using this, the lemma follows: in fact, for $\mathfrak{m} \subset R'$ the
maximal ideal, we find that the map 
$\mathfrak{m}^n \stackrel{1 - M\phi }{\to} \mathfrak{m}^n$ is an isomorphism. 
\end{proof}

\begin{lemma} 
\label{thicksubcatgen}
Let 
$X$ be any qcqs scheme. 
Then $\mathcal{D}^b_{\mathrm{cons}}( X_{\mathrm{et}}, \mathbb{F}_p)$ 
is generated as a thick subcategory by objects of the form 
$f_*(\mathbb{F}_p)$ for $f: Y \to X$ a finite, finitely presented morphism. 
\end{lemma} 
\begin{proof} 
Recall that pushforward by finite, finitely presented morphisms preserves
constructibility 
\cite[Tag 095R]{stacks-project}. 
Without loss of generality, we may assume (in light of the limit formalism
\cite[Tag 01ZA]{stacks-project}) that $X$ is of finite type over
$\mathbb{Z}$, cf.~\cite[Prop.~5.10]{BMarc}. 
By noetherian induction, we may assume that 
the lemma is known for any proper closed subscheme of $X$ and that $X$ is
irreducible (or otherwise we could decompose $X$ into irreducible components). 
It thus suffices to show that if $j: U \subset X$ is the inclusion of an open
subset and $\mathcal{L}$ is a locally constant constructible sheaf on $U$, then $j_! (\mathcal{L})$
belongs to the desired thick subcategory. 
By the ``m\'ethode de la trace''
\cite[Tag 03SH]{stacks-project} 
and \cite[Tag 0A3R]{stacks-project}, 
we find that 
$j_!( \mathcal{L})$
belongs to the thick subcategory generated by $h_!( \mathbb{F}_p)$ for $h: V \to
X$ the
composite of a finite \'etale cover of $U$ together with the inclusion into $X$. 
By Zariski's main theorem \cite[Tag 02LQ]{stacks-project}, we can factor $h$ as the composite of an open
immersion $\widetilde{j}: V \to X'$ together with a finite map $q: X' \to X$. 
It follows that if $Z' \subset X'$  is the complementary closed subscheme to $V
\subset X'$ (with the reduced induced structure), then $h_!(\mathbb{F}_p)$ is
the mapping fiber of $q_*(\mathbb{F}_p) \to (q|_{Z'})_*(\mathbb{F}_p)$, whence
the result. 
\end{proof} 

\begin{theorem}[\cite{EmertonKisin}]
\label{EKthm2}
Suppose $R$ is a regular noetherian $\mathbb{F}_p$-algebra.\footnote{In \cite{EmertonKisin}, the result
is stated when $R$ is smooth over a field; the present extension appears in
\cite{BLRH}, and can also be deduced from the case where $R$ is smooth over
$\mathbb{F}_p$ by passage to filtered colimits and Popescu's theorem.} 
Then the functor
\[ \RHc = \mathrm{RHom}_{\mathcal{D} (\spec(R)_{\mathrm{et}}, \mathbb{F}_p)}( -, \mathbb{G}_a)
: \mathcal{D}^b_{\mathrm{cons}}( \spec(R)_{\mathrm{et}}, \mathbb{F}_p)^{op}
\to \mathcal{D}(R[F]) \]
is fully faithful 
and 
has image given precisely by $\mathcal{D}^b_{\mathrm{fgu}}(R[F])$.
\end{theorem} 
\begin{proof} 
Given $\mathcal{F} \in 
\mathcal{D}^b_{\mathrm{cons}}( \spec(R)_{\mathrm{et}}, \mathbb{F}_p)$, we will
show that 
$\lambda^* \mathcal{F}$ (with notation as in \Cref{bigetale}) is in the essential image of the fully faithful
embedding
$\mathrm{Sol}: \mathcal{D}^{\mathrm{fproj}}_{\mathrm{unit}}(R[F])^{op} \to \mathcal{D}(\sch_R,
\mathbb{F}_p)$ of \Cref{fullfaithfulnessofSol}; more precisely, there exists $M \in
\mathcal{D}^b_{\mathrm{fgu}}(R[F])$ (of finite projective dimension
by \Cref{fgumeansfproj}, and necessarily unique) such that $\lambda^* \mathcal{F} =
\mathrm{Sol}(M)$. It follows that $$M = \mathrm{RHom}_{\mathcal{D}(
\spec(R)_{\mathrm{et}}, \mathbb{F}_p)}( \mathcal{F}, \mathbb{G}_a) = 
\mathrm{RHom}_{\mathcal{D}( \sch_R, \mathbb{F}_p)}( \lambda^* \mathcal{F},
\mathbb{G}_a),$$
 since
$\mathbb{G}_a$ is pushed forward from the big \'etale site; in particular, $M =
\RHc( \mathcal{F})$, which gives full faithfulness of
$\RHc$, with left inverse given by $\lambda_* \mathrm{Sol}$. 

Suppose $B$ is a finite, finitely presented $R$-algebra; let $g: \spec(B) \to
\spec(R)$. 
Then $\lambda^* ( g_* ( \mathbb{F}_p))$ is given by the functor which sends an
$R$-algebra $R'$ to the fiber of $F-1$ on $B \otimes^{\mathbb{L}}_R R'$, thanks to the
Artin--Schreier sequence in the form of \Cref{ASanimatedalg}. 
In light of \Cref{univpropertyunit} (and a resolution, using that $B$ has
finite projective dimension \cite[Lem.~11.3.10]{BLRH}), this is precisely $\mathrm{Sol}$ of the unitalization of
$\mathrm{RHom}_R(B, R)$ equipped with the linear dual of the Frobenius $\phi^* B
\to B$. Since objects of the form $g_*( \mathbb{F}_p)$ generate
$\mathcal{D}^b_{\mathrm{cons}}(\spec(R)_{\mathrm{et}}, \mathbb{F}_p)$ as a
thick subcategory (\Cref{thicksubcatgen}), we
conclude that the same holds with $g_*( \mathbb{F}_p)$
replaced by any object of $\mathcal{D}^b_{\mathrm{cons}}(\spec(R)_{\mathrm{et}},
\mathbb{F}_p)$ as desired.

It remains to identify the essential image of $\RHc$. 
Suppose $M \in \mathcal{D}^b_{\mathrm{fgu}}(R[F])$. We first show that $\mathrm{Sol}(M)$
belongs to the 
image of the pullback $\lambda^*$ from the small \'etale site to the big \'etale
site. 
Without loss of generality, we can suppose that 
$M$ is discrete.
By assumption and \Cref{fguproperties}, $M$ arises as the unitalization of some 
finitely generated $R$-module $M_0$ equipped with a map $M_0 \to \phi^* M_0$. 
Since $M_0$ has finite
projective dimension over $R$ by \cite[Lem.~11.3.10]{BLRH}, 
we find easily (by taking resolutions of $M_0, M_0 \to \phi^* M_0$) that $M$ belongs to the thick subcategory of 
$\mathcal{D}^b_{\mathrm{fgu}}(R[F])$ generated by objects that arise 
as the unitalization of a 
finitely generated \emph{free} $R$-module 
$M_0$ with a map $M_0 \to \phi^* M_0$. We may therefore assume that $M$
itself is obtained as a unitalization of $M_0 \to \phi^* M_0$ with $M_0$ free. 
Thus, $M_0^{\vee}$ is an $R[F]$-module which is 
finitely generated free as an $R$-module. 
Using the universal property of the unitalization again
(\Cref{univpropertyunit}), 
we find that $\mathrm{Sol}(M_0^{\mathrm{unit}})$ is the functor carrying an
$R$-algebra $R'$ to the fiber of $F-1$ on
$R' \otimes_R M_0^{\vee} $. 
This functor evidently commutes with filtered colimits in $R'$. 
The desired rigidity statement (needed to see that $\mathrm{Sol}(M)$ belongs
to the image of $\lambda^*$ as in \Cref{bigetale}) now follows from 
\Cref{rigidity2}. 
Therefore, we conclude $\mathrm{Sol}(M)$ belongs to the image of the pullback $\lambda^*$
from the small \'etale site to the big \'etale site. 

Finally, we need to show that the preimage of $\mathrm{Sol}(M)$ under $\lambda^*$ actually belongs to the
bounded constructible $\infty$-category (and not simply to
$\mathcal{D}(\spec(R)_{\mathrm{et}}, \mathbb{F}_p)$). Boundedness is evident,
since $\mathrm{Sol}(M)$ is bounded as $M$ has finite projective dimension over
$R[F]$ (\Cref{fgumeansfproj}). Moreover, 
constructibility follows because $\mathrm{Sol}(M)$ (as a sheaf on $\sch_R$)
belongs to the thick subcategory generated by $\prod_I \mathbb{G}_a$, whence is
pseudo-coherent in $\mathcal{D}(\sch_R, \mathbb{F}_p)$. This implies that the
preimage of $\mathrm{Sol}(M)$ is compact in any truncation of
$\mathcal{D}(\spec(R)_{\mathrm{et}}, \mathbb{F}_p)$, whence the cohomology
sheaves are constructible  (cf.~\Cref{compactobjofderivedcat} and the references
in the proof). 
\end{proof} 
Note that the proof also establishes that for $M \in
\mathcal{D}^b_{\mathrm{fgu}}(R[F])$, we have a natural equivalence 
$$\RHc( \lambda_*( \mathrm{Sol}(M))) \simeq 
\mathrm{RHom}_{\mathcal{D}( \sch_R, \mathbb{F}_p)}( \lambda^* \lambda_*
\mathrm{Sol}(M), \mathbb{G}_a) = 
\mathrm{RHom}_{\mathcal{D}( \sch_R, \mathbb{F}_p)}( 
\mathrm{Sol}(M), \mathbb{G}_a)\simeq 
M,$$ 
in light of 
\Cref{fullfaithfulnessofSol} (and as proved in \cite{EmertonKisin}).

\appendix

\section{A proof of Breen's theorem}

In the appendix, we describe a proof of the theorem of \cite{Breen} that is
independent of explicit computations in algebraic topology
(cf.~\cite[Lem.~0.3]{FranjouLannesSchwartz} and \cite[Cor.~1.2]{KuhnIII} for
other algebraic proofs). The essential idea of the
proof (the use of $v$-descent results of \cite{BSWitt}) was suggested by
Scholze. Let us first recall the statement. 

\begin{theorem}[Breen \cite{Breen}] 
\label{Breenthmappendix}
Let $R$ be a perfect $\mathbb{F}_p$-algebra. Then
$\mathrm{Ext}^i_{\mathcal{D}(\schpp_R, \mathbb{F}_p)}(
\mathbb{G}_a, \mathbb{G}_a) = 0$ for $i > 0$ and is given by 
$R[F^{\pm 1}]$ for $i = 0$. 
\end{theorem} 

The strategy is as follows. It suffices to replace the source $\mathbb{G}_a$ by the big
Witt vectors, $W^{\mathrm{big}}$ and to compute the Ext groups in
$\mathcal{D}(\schpp_R)$. 
In fact, this follows from the equivalence
\( \mathrm{RHom}_{\mathcal{D}(\schpp_R)}( W,
\mathbb{G}_a) = 
\mathrm{RHom}_{\mathcal{D}(\schpp_R, \mathbb{F}_p)}( \mathbb{G}_a,
\mathbb{G}_a)
\) and the fact that $W$ is a retract of $W^{\mathrm{big}}$. 
It will be convenient to use big Witt vectors for the
following reason: $W^{\mathrm{big}}$ can be approximated by  the subsheaf
$W^{\mathrm{rat}}$ of \emph{rational} Witt vectors, for which we can do the calculation explicitly
since, in the $v$-topology, $W^{\mathrm{rat}}$ is essentially the free abelian
group on $\mathbb{A}^1$ (with a basepoint at zero). This approach does not
recover the more general calculation of extensions of the additive group on all
(not necessarily perfect) $\mathbb{F}_p$-schemes of \cite{Breen78}.

Generalizations of 
\Cref{Breenthmappendix} have also been obtained in the relatively perfect site by
Kato \cite{Kato1} and by Ansch\"utz--Le Bras \cite{ALBFourier} in analytic (perfectoid) settings.

\subsection{Comparing rational and big Witt vectors}
Fix a base field $k$ (which will later be taken to be $\mathbb{F}_p$). 
Let $\sch_{k}$ be the site of all 
(qcqs) $k$-schemes, equipped with the \'etale topology. 
In this subsection, we compare the sheaves of rational Witt vectors and big Witt
vectors on $\sch_{k}$.

\begin{definition}[Ind-affine schemes]
Let $\mathrm{AffSch}^{\ast}_{k, \mathrm{ft}}$ denote the category of
pointed affine
schemes of finite type over $k$ (i.e., the opposite to the category of finite type
augmented $k$-algebras). We consider the ind-category
$\mathrm{Ind}( \mathrm{AffSch}^{\ast}_{k, \mathrm{ft}})$. 
Now
$\mathrm{AffSch}^{\ast}_{k, \mathrm{ft}}$ admits finite coproducts and finite
products, and finite products distribute over finite coproducts. 
It follows that 
$\mathrm{Ind}( \mathrm{AffSch}^{\ast}_{k, \mathrm{ft}})$
admits filtered colimits (and arbitrary coproducts) and finite products, and finite products distribute
over filtered colimits (and arbitrary coproducts). 

We have a natural functor \begin{equation} \label{functor} \mathrm{Ind}( \mathrm{AffSch}^{\ast}_{k,
\mathrm{ft}}) \to \mathrm{Shv}(\sch_{k}, \mathrm{Set})\end{equation} carrying the
class of $X \in \mathrm{AffSch}^{\ast}_{k, \mathrm{ft}}$ to the
representable sheaf and then is defined by Kan extension; this construction preserves filtered colimits and
finite products. 
Typically, we will abuse notation and 
refer to an object of 
$\mathrm{Ind}( \mathrm{AffSch}^{\ast}_{k,
\mathrm{ft}})$ and the corresponding sheaf on $\sch_{k}$ by the same notation. 
\end{definition}

\begin{construction}[Commutative monoids in ind-affine schemes]
\label{commmonoid}
We consider commutative monoids in 
$\mathrm{Ind}( \mathrm{AffSch}^{\ast}_{k, \mathrm{ft}})$. 
Note that any such yields a commutative monoid in
$\mathrm{Shv}(\sch_{k}, \mathrm{Set})$, since the functor of \eqref{functor}
preserves finite products.  

Given  $X \in 
\mathrm{AffSch}^{\ast}_{k, \mathrm{ft}}$, 
we can form the free commutative monoid on $X$ in
the category $\mathrm{Ind}(\mathrm{AffSch}^{\ast}_{k, \mathrm{ft}})$ (note that
this forces the basepoint to be the monoid unit); this is the filtered colimit
$\mathrm{free}_{\ast}(X)\stackrel{\mathrm{def}}{=}\varinjlim_i \mathrm{Sym}^i X$ where the
inclusion maps are induced by $\ast$ (and the colimit is taken in the
ind-sense). 
\end{construction}

Now we specialize to the example of interest. 
\begin{construction}[The free commutative monoid on $\mathbb{A}^1_{k}$]
We let $\mathbb{A}^1_{k}$ denote the affine line over $k$
with basepoint at origin. 
The $i$th symmetric power
$\mathrm{Sym}^i 
\left(\mathbb{A}^1_{k}\right), i \geq 0 $ is given as $ \spec
k[s_1, \dots, s_i]$, for $s_1, \dots, s_i$
the elementary symmetric functions. 
Using the basepoint at the origin, we can 
form the object of 
$\mathrm{Ind}( \mathrm{AffSch}^{\ast}_{k, \mathrm{ft}})$ given as
$W_{\mathrm{rat}}^+ = \free_{\ast}( \mathbb{A}^1_k) =  \varinjlim_i {\mathrm{Sym}^i
\left(\mathbb{A}^1_{k}\right)} \in 
\mathrm{Ind}( \mathrm{AffSch}^{\ast}_{k, \mathrm{ft}})$. 

Unwinding the definitions and using the theorem on symmetric functions, we see that 
for any $k$-algebra $R$, the commutative monoid
$W_{\mathrm{rat}}^+(R)$ is $1 + tR[t]$ under
multiplication.\footnote{A variant of this fact is used essentially in the
Cartier--Dieudonn\'e theory of commutative formal groups,
cf.~\cite[Th.~2.2]{Chai}.} 
Given a ring $R$, the group completion of $W_{\mathrm{rat}}^+(R)$ is known as the 
\emph{rational Witt vectors} of $R$, cf.~\cite{Almkvist1, Almkvist2}, and the notation 
$W_{\mathrm{rat}}^+$ is chosen for this reason. 
In the following, we let $W_{\mathrm{rat}}$ denote the group completion of
$W_{\mathrm{rat}}^+$ as sheaves on $\sch_{k}$ (so $W_{\mathrm{rat}}$
is the \'etale sheafification of the rational Witt vectors). 
\end{construction}

\begin{construction}[Big Witt vectors] 
We let $W^{\mathrm{big}}$ denote the big Witt vector functor
(cf.~\cite[Sec.~17]{Hazewinkel} for an account); it
carries a commutative ring $R$ to the abelian group $1 + t R[[t]]$ under
multiplication, and is representable by a polynomial ring on countably many
variables. We have a natural map $W_{\mathrm{rat}}^+ \to W^{\mathrm{big}}$
including polynomials inside power series. 
\end{construction} 

The purpose of this subsection is to prove the following result to the effect that maps into 
$\mathbb{G}_a$ (or any direct sum of copies of such) does not see the difference
between $W_{\mathrm{rat}}$ and $W^{\mathrm{big}}$; the argument was inspired by the
Clausen--Scholze solid
formalism \cite{condensed}. 
\begin{proposition} 
\label{mainsubsectionthm}
The map 
\( W_{\mathrm{rat}} \to W^{\mathrm{big}}  \)
in $\mathcal{D}( \sch_{k})$ induces an equivalence upon applying
$\rhom_{\mathcal{D}( \sch_{k})}	(-, \mathbb{G}_a \otimes_k V)$ for any
$k$-vector space $V$. 
\end{proposition}

\begin{remark}[Tensors with pointed sets] 
In the following, we will use that 
any pointed category with finite coproducts (such as the categories
$\mathrm{Ind}( \mathrm{AffSch}^{\ast}_{k, \mathrm{ft}})$ or commutative monoids
in it, or the
category of sheaves of abelian monoids on a site) is
naturally tensored over the category of finite pointed sets; we denote this
tensor by $\otimes$.  Explicitly, given a pointed set and an object $X$, the
tensor $S \otimes X$ is the pushout of the coproduct $\bigsqcup_S X $ along $X
\to \ast$ (where $X$ includes in $\bigsqcup_S X$ via the basepoint of
$S$).\footnote{Note that if $S = S_0 \sqcup \ast$, then $S \otimes
X = \bigsqcup_{S_0} X$; thus the above construction only requires finite
coproducts and not pushouts.} 
\end{remark} 

In the next result, we identify an object of $\mathrm{Ind}(
\mathrm{AffSch}^{\ast}_{k, \mathrm{ft}})$ with the corresponding sheaf on
$\sch_{k}$, so that we can talk about the cohomology of a sheaf on
an object of $\mathrm{Ind}(
\mathrm{AffSch}^{\ast}_{k, \mathrm{ft}})$. 

\begin{proposition} 
Let $A$ be an augmented $k$-algebra of finite type, so $\spec(A)
\in \mathrm{AffSch}^{\ast}_{k, \mathrm{ft}}$. 
Let $I_A$ denote the augmentation ideal, and let $V$ be any $k$-vector space. 

For any pointed finite set $S$, there is a natural isomorphism
\begin{equation} 
R \Gamma( S \otimes \mathrm{free}_{\ast}(\spec(A)), \mathbb{G}_a \otimes_k V) 
\simeq \prod_{i \geq 0} \Gamma^i ( \hom_{k}(
k[\overline{S}], I_A)) \otimes_k V, 
\end{equation} 
for $k[\overline{S}]$ the cokernel of the map $k \to
k[S]$ arising from the basepoint and $\Gamma^i$ denoting the $i$th divided power
functor. 
\label{forallA}
\end{proposition} 
\begin{proof} 
By the universal property, the commutative monoid $S \otimes \mathrm{free}_*(\spec(A))$ can equivalently 
be described as $\free_*( S \otimes \spec(A))$. 
Now $S \otimes \spec(A)$, the tensor of $S$ with $\spec(A)$ in the category of pointed affine finite type
$k$-schemes, is exactly the spectrum of the limit of the diagram
\[ \xymatrix{ 
& & k \ar[d]  \\
& A^S  \ar[d] \ar[r] & A  \\
k \ar[r] &  k^S
},\]
whose augmentation ideal is exactly $\hom_{k}(
k[\overline{S}], I_A)$. 
It follows (cf.~\Cref{commmonoid}) that 
$S \otimes \mathrm{free}_*(\spec(A))$
is computed in 
$\mathrm{Ind}( \mathrm{AffSch}^{\ast}_{k, \mathrm{ft}})$ as the filtered colimit of
$\sym^i(\spec ( k \oplus \hom_{k}( k[\overline{S}],
I_A)))$, whence the result. 
\end{proof} 

In the next results, we let $x$ denote the coordinate on
$\mathbb{A}^1_{k}= \spec k[x]$, and let
$k[x]^+$ denote the augmentation ideal $(x) \subset k[x]$.  
The following result 
is a consequence of \Cref{forallA} for the pointed affine line. 
\begin{proposition} 
For any pointed finite set $S$ and $k$-vector space $V$, there is a natural isomorphism (with notation as
in \Cref{forallA})
\begin{equation}  \label{fnsonwrat} R \Gamma( S \otimes W_{\mathrm{rat}}^+,
\mathbb{G}_a \otimes_k V) \simeq
 \prod_{i \geq 0} \Gamma^i (\hom_{k}( k[\overline{S}],  
k[x]^+)) \otimes_k V.
\end{equation}
\end{proposition}

\begin{proposition} 
For any pointed finite set $S$, and $k$-vector space $V$, there
is a natural isomorphism (with notation as in \Cref{forallA})
\begin{equation}  
\label{Wbigfunctions}
R \Gamma(S \otimes W^{\mathrm{big}}, \mathbb{G}_a \otimes_k V) = \bigoplus_{i
\geq 0} \Gamma^i (\hom_{k}( k[\overline{S}],  
k[x]^+)) \otimes_k V. \end{equation}
With this identification, the map $S \otimes W^{\mathrm{rat}} \to S \otimes
W^{\mathrm{big}}$
induces the obvious sum-to-product map on $R \Gamma(-, \mathbb{G}_a \otimes_k V)$. 
\end{proposition} 

\begin{proof} 
Since $S \otimes W^{\mathrm{big}}$ is affine, we may assume $V = k$
throughout. 
The map in question is the map 
$$\varinjlim_n  \sym^n( S \otimes \mathbb{A}^1_k ) \to S \otimes W^{\mathrm{big}}.$$
Note that both sides carry $\mathbb{G}_m$-actions. 
We claim that the map 
$\sym^n( S \otimes \mathbb{A}^1_k ) \to S \otimes W^{\mathrm{big}}$ (of affine schemes)
induces an isomorphism in weights $\leq n$ on functions, i.e., $R \Gamma(-,
\mathbb{G}_a)$.
It follows that functions on $S \otimes W^{\mathrm{big}}$ have a natural grading, and
are the inverse limit (over $n$) in graded vector spaces of functions on $\sym^n(
S \otimes \mathbb{A}^1_k)$. From this, the result easily follows. 
To see the claim, we may reduce to the case where $S = \left\{\ast, 1\right\}$
(via taking tensor products), for which we need to see that $\mathrm{Sym}^n (
\mathbb{A}^1_k) \to W^{\mathrm{big}}$ induces an isomorphism in weights $ \leq n$ on
functions. 
This is evident given the explicit description of functions on
$W^{\mathrm{big}}$ as the polynomial ring in all symmetric functions;
alternatively, $\mathrm{Sym}^n( \mathbb{A}^1_k) \to W^{\mathrm{big}}$ is the
natural transformation 
that sends $R$ to the map 
$(1 + t R[t])_{\leq n} \to 1 + t R[[t]]$. 
\end{proof} 

Given a pointed simplicial set $S_\bullet$, we can (by taking geometric
realizations) consider $S_\bullet \otimes W_{\mathrm{rat}}^+, S_\bullet \otimes
W^{\mathrm{big}}$ as sheaves of animated commutative monoids. 

\begin{proposition} 
\label{functionsonGa}
For any simplicial pointed finite set $S_\bullet$ which is connected and any
$k$-vector space $V$, the map 
$S_\bullet \otimes W_{\mathrm{rat}}^+  \to S_\bullet \otimes W^{\mathrm{big}}$ 
induces an equivalence on $R \Gamma( -, \mathbb{G}_a \otimes_k V)$.  
\end{proposition} 
\begin{proof} 
Recall that both infinite direct sums and direct products commute with
totalizations in the coconnective derived $\infty$-category. 
Therefore, we find that 
the equations \eqref{fnsonwrat} and 
\eqref{Wbigfunctions} hold for $R \Gamma(S_\bullet \otimes W_{\mathrm{rat}}^+,
\mathbb{G}_a \otimes_k V), 
R \Gamma(S_\bullet \otimes W^{\mathrm{big}},
\mathbb{G}_a \otimes_k V)$ with $k[\overline{S}]$ replaced by reduced singular
chains on $S_\bullet$, i.e., $\overline{C_*}(S_\bullet; k)$ (and with
the derived functors of the divided powers). 
Now the result follows because if $U \in \mathcal{D}(k)^{\geq 1}$,
then $\Gamma^i U \in \mathcal{D}(k)^{\geq i}$; for example, this
follows by duality (and passage to filtered colimits) for the corresponding
connectivity assertion for nonabelian derived symmetric powers in the
connective case, cf.~\cite[Prop.~25.2.4.1]{SAG}. 
Therefore, because of the connectivity bounds, the map from the direct sum to
the direct product 
(i.e., from the analogs of \eqref{Wbigfunctions} and \eqref{fnsonwrat}) is an
equivalence. 
\end{proof} 

\begin{proof}[Proof of \Cref{mainsubsectionthm}]
We use throughout that $B W_{\mathrm{rat}}^+ = BW_{\mathrm{rat}}$ since
$W_{\mathrm{rat}}$ is the group completion of the commutative monoid
$W_{\mathrm{rat}}^+$; this is a special case of the group-completion
theorem \cite{McDS}. 
From \Cref{functionsonGa} (applied to a simplicial model of a wedge of circles), 
it follows that for any finite pointed set $T$, 
the map of classifying stacks $B( T \otimes W_{\mathrm{rat}}^+) \to B( T \otimes
W^{\mathrm{big}})$ induces an equivalence on $R\Gamma(-, \mathbb{G}_a \otimes_k
V)$. 
Note that both 
$B( T \otimes W_{\mathrm{rat}}^+) \to B( T \otimes
W^{\mathrm{big}})$
are abelian group objects. 
By the functorial Breen--Deligne resolution \cite[Lec.~4]{condensed}, there is a resolution of
$W_{\mathrm{rat}}[1]$ in $\mathcal{D}(\sch_{k})$ all of whose terms 
are finite direct sums of the form $\mathbb{Z}[ B( T \otimes W_{\mathrm{rat}}^+) ]$ for various
finite pointed sets $T$, and an analogous resolution of 
$W^{\mathrm{big}}[1]$ all of whose terms are of the form 
 $\mathbb{Z}[ B( T \otimes W^{\mathrm{big}}) ]$ (with the same finite sets and finite direct
 sums appearing). Since $\rhom_{\mathcal{D}(\sch_{k})}(-,
 \mathbb{G}_a \otimes_k V)$ carries the map 
 $\mathbb{Z}[ B( T \otimes W_{\mathrm{rat}}^+) ] \to 
\mathbb{Z}[ B( T \otimes W^{\mathrm{big}}) ]$ to an equivalence by
\Cref{functionsonGa}, the
result follows. 
\end{proof} 

\subsection{Proof of Breen's theorem}

\newcommand{\schpph}{\mathrm{Sch}^{\mathrm{perf,} v}}

\begin{construction}[The $v$-site of perfect schemes] 
Let $\schpph_{\mathbb{F}_p}$ denote the site of qcqs perfect
$\mathbb{F}_p$-schemes, equipped with the $v$-topology (or universally
subtrusive topology) of \cite{Rydh, BSWitt}. Note that this is a subcanonical
topology since we are working with perfect schemes, cf.~\cite[Th.~5.17]{BMarc}. 
In particular, the additive $\mathbb{G}_a$ is a sheaf on
$\schpph_{\mathbb{F}_p}$, with trivial higher cohomology on affines, 
cf.~\cite[Th.~11.2(2)]{BSWitt}. 

We have a morphism of sites
$\eta: \schpph_{\mathbb{F}_p} \to \sch_{\mathbb{F}_p}$  arising from the
inverse limit perfection functor $(-)^{\mathrm{perf}}: \sch_{\mathbb{F}_p} \to
\schpph_{\mathbb{F}_p}$, inducing a geometric
morphism of topoi. 
\end{construction}

\begin{remark}[Examples of $\eta$-pullbacks] 
First, $\eta^* W^{\mathrm{big}}$ is simply (by pulling back the representing
object) the big Witt vector functor on
$\schpp_{\mathbb{F}_p}$, which is already a $v$-sheaf. 
Second, note that $\eta^* W_{\mathrm{rat}}^+$ is the functor carrying a perfect $\mathbb{F}_p$-algebra $R$ to $1 + tR[t]$,
considered as a monoid under multiplication; this is already a $v$-sheaf
(since as a set it is a
filtered colimit of finite products of copies of $R$). 
\end{remark} 

The next result, which is closely related to \cite[Lem.~5.16]{SuslinVoevodsky},
states that quotients of affine schemes by finite group actions are actually
quotients in the $v$-topology. 
\begin{proposition} 
Let $R$ be a perfect $\mathbb{F}_p$-algebra with an action of a finite group
$G$; let $Y = \spec(R)$ and $X = \spec(R^G)$. 
Then the map $Y \to X$ exhibits, on representables $h_{Y}
\to h_{X}$, the target as the quotient by $G$ in the category of
sheaves of sets on $\schpph_{\mathbb{F}_p}$. 
\label{quotienthsheafofsets}
\end{proposition} 
\begin{proof} 
Since $R^G \subset R$ is an integral extension, 
the map of schemes $Y \to X$ is integral and surjective, hence a $v$-cover
\cite[Rem.~2.5]{Rydh}.  
It follows that the map $h_Y/G \to h_X$ (of sheaves of sets) is surjective. %Note
%that it is an isomorphism on $K$-points for $K$ any algebraically closed field
%of characteristic $p$, cf.~\cite[Th.~2,
%Sec.~V.2.2]{BourbakiCA}. 
To see that $h_{Y}/G \to h_X$ is an isomorphism, it suffices to show that the map 
$G \times Y \to Y \times_X Y$
is a $v$-cover, but this follows because it is also a surjective integral map
by \cite[Th.~2, Sec.~V.2.2]{BourbakiCA}. 
\end{proof} 

\begin{proposition} 
\label{wratfreecommonoid}
The object 
$\eta^* W_{\mathrm{rat}}^+ \in \mathrm{Shv}(
\schpph_{\mathbb{F}_p}, \mathrm{CMon})$ is the free commutative monoid  on
the pointed object $h_{\mathbb{A}^{1, \mathrm{perf}}}$ (with basepoint at the
origin). That is, for any sheaf of commutative monoids 
$M$ on 
$\schpph_{\mathbb{F}_p}$, we have
\[ \hom_{\mathrm{Shv}(
\schpph_{\mathbb{F}_p}, \mathrm{CMon})} ( \eta^* W_{\mathrm{rat}}^+, M ) \simeq
\hom_{\mathrm{Shv}(
\schpph_{\mathbb{F}_p}, \mathrm{Set}_{\ast})} ( h_{\mathbb{A}_{\mathbb{F}_p}^{1,
\mathrm{perf}}}, M )
.  \]
\end{proposition} 
\begin{proof} 
The $v$-sheaf $\eta^*W^+_{\mathrm{rat}}$ is the 
filtered colimit $\varinjlim h_{\mathrm{Sym}^i( \mathbb{A}^{1,
\mathrm{perf}}_{\mathbb{F}_p})}$. However, these symmetric powers can be either
taken at the level of schemes or at the level of $v$-sheaves of sets 
thanks to \Cref{quotienthsheafofsets}, whence the result since the latter
construction gives the free commutative monoid sheaf on $h_{\mathbb{A}^{1,
\mathrm{perf}}_{\mathbb{F}_p}}$. 
\end{proof}

\begin{proof}[Proof of \Cref{Breenthmappendix}] 
The claim for $i = 0$ is straightforward (cf.~\cite[\S 1.4]{Breen}) to deduce
from the claim about endomorphisms of $\mathbb{G}_a$ on all
$\mathbb{F}_p$-algebras, so we treat the vanishing for $i > 0$. 
The 
derived pushforward of $\mathbb{G}_a$ from the $v$-site to the \'etale site on
perfect schemes is
simply 
$\mathbb{G}_a$ thanks to \cite[Th.~11.2(2)]{BSWitt}. 
Thus, by adjunction, 
it suffices to show that 
$\rhom_{\mathcal{D}(\schpph_{\mathbb{F}_p}, \mathbb{F}_p)}(\mathbb{G}_a,
\mathbb{G}_a)$ is concentrated in cohomological degree zero. 
We will show that 
$\rhom_{\mathcal{D}(\schpph_{\mathbb{F}_p})}(W^{\mathrm{big}}, \mathbb{G}_a)$ is concentrated in
degree zero. This will suffice to prove the result, because  $W$ is a retract of $W^{\mathrm{big}}$ and $\mathbb{G}_a =
W/p$ on 
$\schpph_{\mathbb{F}_p}$.

By \Cref{wratfreecommonoid} and group completion,
we find that $\eta^* W_{\mathrm{rat}}(-)$ is the free abelian group sheaf on 
$\mathbb{A}^{1, \mathrm{perf}}_{\mathbb{F}_p}$ with the origin as zero. 
Thus, also 
using \cite[Th.~11.2(2)]{BSWitt}
again, 
we have that 
$$\rhom_{\mathcal{D}( \schpph_{\mathbb{F}_p})}(\eta^* W_{\mathrm{rat}},
\mathbb{G}_a) = \mathrm{fib} \left( R \Gamma_v( \mathbb{A}^{1, {\mathrm{perf}}}_{\mathbb{F}_p},
\mathbb{G}_a) \to \mathbb{F}_p\right) = 
\left\{f \in \mathbb{F}_p[x^{1/p^\infty}], f(0) = 0\right\}
$$ is discrete. 
Thus, it suffices to show that the map 
$\eta^* W_{\mathrm{rat}} \to \eta^* W^{\mathrm{big}} = W^{\mathrm{big}}$ 
induces an equivalence on $\rhom_{\mathcal{D}( \schpph_{\mathbb{F}_p})}( -,
\mathbb{G}_a)$. 
But this follows because 
the map $W_{\mathrm{rat}} \to W^{\mathrm{big}}$ induces an equivalence on
$\rhom_{\mathcal{D}(\sch_{\mathbb{F}_p})}( -, \eta_*\mathbb{G}_a)$, noting that
$\eta_* \mathbb{G}_a = \varinjlim_{\phi} \mathbb{G}_a$
(again by \cite[Th.~11.2(2)]{BSWitt})
and using
\Cref{mainsubsectionthm} (applied with $V$ a countably-dimensional
$\mathbb{F}_p$-vector space, since $\varinjlim_{\phi} \mathbb{G}_a$ is a
mapping cone of an endomorphism of $\bigoplus_{i = 0}^\infty \mathbb{G}_a$). 
\end{proof}

We also record an equivalent form of \Cref{Breenthmappendix}, where the 
extension groups are taken in the big \'etale site (rather than the perfect
site). 
For this, we consider the inverse limit perfection $\varprojlim_{\phi}
\mathbb{G}_a$ in 
$\mathcal{D}( \sch_{\mathbb{F}_p})$ and calculate maps into $\mathbb{G}_a$. 

\begin{corollary} 
\label{variantBreenthm}
Let $R$ be a perfect $\mathbb{F}_p$-algebra and let $I$ be any set. 
We have 
$$\rhom_{\mathcal{D}( \sch_R, \mathbb{F}_p)}( \prod_I
\varprojlim_\phi \mathbb{G}_a, \mathbb{G}_a) = \bigoplus_I R[F^{\pm 1}],$$ where
the products are calculated in 
$\mathcal{D}( \sch_R, \mathbb{F}_p)$. 
\end{corollary} 
\begin{proof} 
We may replace the \'etale topology on $\sch_R$ with the flat
topology since $\mathbb{G}_a$ satisfies flat descent; then
$\prod_I \varprojlim_{\phi} \mathbb{G}_a$ becomes a discrete object
(representable by a product of copies of the perfection of $\mathbb{G}_a$). 
Using the Breen--Deligne resolution \cite[Lec.~4]{condensed} again in the flat
site, one sees that 
$$\rhom_{\mathcal{D}( \sch_R, \mathbb{F}_p)}( \prod_I \varprojlim_{\phi} \mathbb{G}_a, \mathbb{G}_a) = 
\rhom_{\mathcal{D}( \schpp_R, \mathbb{F}_p)}( \prod_I
\mathbb{G}_a, \mathbb{G}_a)
= \bigoplus_I 
\rhom_{\mathcal{D}( \schpp_R, \mathbb{F}_p)}( \mathbb{G}_a, \mathbb{G}_a),$$
since (for the second claim) $\mathbb{G}_a$-cohomology carries cofiltered limits of qcqs schemes along
affine transition maps to filtered colimits. 
\end{proof}

\bibliographystyle{amsalpha}
\bibliography{relativelyperfect}

\end{document}